\documentclass[a4paper,10pt]{amsart}
\usepackage[left=25mm, right=25mm, top=15mm, bottom=15mm, includeheadfoot]{geometry}	

\usepackage[T1]{fontenc}

\usepackage[leqno]{amsmath}
\usepackage{amsthm}
\usepackage{amsfonts}
\usepackage{amssymb}

\usepackage[utf8]{inputenc}
\usepackage[english]{babel}

\theoremstyle{plain}
\newtheorem{lem}{Lemma}[section]
\newtheorem{thm}{Theorem}

\theoremstyle{remark}
\newtheorem{rem}[lem]{Remark}

\theoremstyle{definition}
\newtheorem{defi}[lem]{Definition}

\providecommand{\abs}[1]{\lvert#1\rvert} 
\providecommand{\norm}[1]{\lVert#1\rVert}

\DeclareMathOperator{\ud}{\!d\!}
\DeclareMathOperator{\Ud}{D}

\DeclareMathOperator{\supp}{supp}
\DeclareMathOperator{\esssup}{ess\, sup}
\DeclareMathOperator{\Div}{div}
\DeclareMathOperator{\Rot}{rot}
\DeclareMathOperator{\Dt}{\frac{\ud}{\ud t}}
\DeclareMathOperator{\const}{const}
\DeclareMathOperator{\Id}{Id}

\numberwithin{equation}{section} 

\begin{document}

\title[Strong solutions to micropolar equations]{Large time existence of strong solutions to micropolar equations in cylindrical domains}
\author{Bernard Nowakowski}

\thanks{The author is partially supported by Polish KBN grant N N201 393137}
\address{Bernard Nowakowski\\ Institute of Mathematics\\ Polish Academy of Sciences\\ \'Snia\-deckich 8\\ 00-956 Warsaw\\ Poland}
\email{bernard@impan.pl}


\date{August, 2011}

\begin{abstract}
	We examine the so-called micropolar equations in three dimensional cylindrical domains under Navier boundary conditions. These equations form a generalization of the ordinary incompressible Navier-Stokes model, taking the structure of the fluid into account. We prove that under certain smallness assumption on the rate of change of the initial data and the external data there exists a unique and strong solution for any finite time $T$. 
\end{abstract}

\maketitle

\section{Introduction}

Micropolar equations, which were suggested and introduced by A. Eringen in 1966 (see \cite{erin}), are a significant step toward generalization of the standard Navier-Stokes model, which describes motion of viscous and incompressible fluid. By their very nature, the equations suggested by C.L. Navier and shortly afterward complemented and mathematically formalized by G.G. Stokes, do not take into account the structure of the media they describe, although it plays crucial role in modeling for some well-known fluids, e.g. animal blood or liquid crystals (see \cite{pop1, pop2}). The deviancy becomes highly apparent in microscales. 

In this wok we base on the model proposed by A. Eringen in \cite{erin}. It takes into account that the molecules may rotate independently of the fluid rotation. Thus, the standard Navier-Stokes system is complemented with another vector equation and the whole system of equations is given by

\begin{equation}
	\begin{aligned}\label{p1}
		&v_{,t} + v\cdot \nabla v - (\nu + \nu_r)\triangle v + \nabla p = 2\nu_r\Rot\omega + f   & &\text{in } \Omega^{t} := \Omega\times(t_0,t),\\
		&\Div v = 0 & &\text{in } \Omega^{t},\\
		&\begin{aligned}
			&\omega_{,t} + v\cdot \nabla \omega - \alpha \triangle \omega - \beta\nabla\Div\omega + 4\nu_r\omega \\
			&\mspace{60mu} = 2\nu_r\Rot v + g
		\end{aligned}& &\text{in } \Omega^{t},
    \end{aligned}
\end{equation}
where $\alpha = c_a + c_d$, $\beta = c_0 + c_d - c_a$ and the viscosity coefficients $\nu$, $\nu_r$, $c_0$, $c_a$, $c_d$ are constant and positive. The unknowns are the velocity field $v = (v_1(x,t),v_2(x,t),v_3(x,t)) \in \mathbb{R}^3$, the microrotation field $\omega = (\omega_1(x,t),\omega_2(x,t),\omega_3(x,t)) \in \mathbb{R}^3$ and the pressure $p = p(x,t) \in \mathbb{R}^1$. 

To consider the above problem as initial-boundary we need to specify the domain, the boundary and the initial conditions. We shall now describe them in greater detail.

\subsection*{The domain}
To describe the domain $\Omega$ consider a closed curve $\varphi\colon \mathbb{R}^2 \to \mathbb{R}$, $\varphi(x_1,x_2) = \const$ which is at least of class $\mathcal{C}^2$. We do not put any additional constraints on the properties or on the shape of $\varphi$. Note that if $\varphi$ has a center of symmetry, then it is more natural to investigate problem \eqref{p1} in cylindrical coordinate system $(r,\varphi,z)$ (see e.g. \cite{wm4}, \cite{wm7}). However, such an approach results in technical difficulties related to weighted spaces. 

Having $\varphi$ defined we define the domain $\Omega$ as a set 
\begin{equation*}
	\left\{(x_1,x_2)\in \mathbb{R}^2\colon \varphi(x_1,x_2) \leq c_0\right\}\times\left\{x_3\colon -a \leq x_3 \leq a\right\},
\end{equation*}
where the constant $a > 0$. It is clear that $\Omega$ is a finite and regular pipe placed alongside the $x_3$-axis. 

The boundary $\partial \Omega$ will be denoted by $S$ for convenience. This set is composed of two sets, $S_1$ and $S_2$, $S = S_1 \cup S_2$, where by $S_1$ we denote the side boundary and by $S_2$ the top and the bottom of the cylinder. Thus
\begin{align*}
    S_1 &= \{x\in \mathbb{R}^3\colon\varphi(x_1,x_2)=c_0,\ -a<x_3<a\}\\
\intertext{and}
    S_2 &= \{x\in \mathbb{R}^3\colon\varphi(x_1,x_2)<c_0, \text{ $x_3$ is equal either to $-a$ or to $a$}\}.
\end{align*}

We draw the distinction between $S_1$ and $S_2$, because some boundary conditions to auxiliary problems will be expressed by significantly different formulas.

By simple computation we immediately get
\begin{equation}\label{p4}
	\begin{aligned}
		&n\vert_{S_1} = \frac{1}{\abs{\nabla \varphi}}(\varphi_{,x_1},\varphi_{,x_2},0) & &\tau_1\vert_{S_1} = \frac{1}{\abs{\nabla \varphi}}(-\varphi_{,x_2},\varphi_{,x_1},0) & & \tau_2\vert_{S_1} = (0,0,1)\\
		&n\vert_{S_2} = \left(0,0,\frac{a}{\abs{a}}\right) & &\tau_1\vert_{S_2} = (1,0,0) & &\tau_2\vert_{S_2} = (0,1,0),
	\end{aligned}
\end{equation}
where $n$, $\tau_i$, $i = 1,2$, are the unit outward normal and the unit tangent vectors respectively.

Let us now justify the choice of the domain. Since the problem of uniqueness (or regularity) of weak solutions for Navier-Stokes equations in three dimensions is open, several alternative approaches were taken. One of them is intensely focused upon search for such solutions that are close to two dimensional (see e.g. \cite{ren1}, \cite{wm2}, \cite{wm6}). It is also our case. Therefore, the solutions which are proved to exist, can be regarded as a slight perturbation of two dimensional micropolar flow along the perpendicular direction. This perturbation will be somehow measured by $\delta(t)$ (see \eqref{eq26}), which we introduce later. We shall emphasize that since we only require the initial rate of change of the flow and microrotation, as well as the derivatives of the external data with respect to $x_3$ to be small, the flow alongside the cylinder can be large, but close to constant. 

\subsection*{The boundary and initial conditions}

We supplement system \eqref{p1} with the boundary conditions of the form
\begin{equation}
		\begin{aligned}\label{p2}
			&v\cdot n = 0 & &\text{on $S^{\infty} := S\times(t_0,t)$}, \\
			&\Rot v \times n = 0  & &\text{on $S^{t}$}, \\
			&\omega = 0 & &\text{on $S_1^{t}$}, \\
			&\omega' = 0, \qquad \omega_{3,x_3} = 0 & &\text{on $S_2^{t}$},
    \end{aligned}    
\end{equation}
where $n$ is the unit outward vector. 

Let us briefly justify our choice. Our proof of the existence of regular solutions to \eqref{p1} uses an estimate for the third component of the vorticity field (see Lemmas \ref{lem9} and \ref{lem7}) which is equal to $v_{2,x_1} - v_{1,x_2}$. Thus, the Dirichlet condition $v = 0$ is not admissible because it does not provide any information concerning derivatives of $v$ on the boundary. Therefore we decided to employ slip boundary condition $\Rot v \times n = 0$, which was already considered by C.L. Navier in 1827 (and is often referred as the Navier boundary condition; see \cite{nav}) and satisfies the relation (see Lemma \ref{lem32}) $n\cdot \mathbb{D}(v) \cdot \tau_{\alpha} + 2 \kappa (2 - \alpha)(v\cdot \tau_1) = \Rot v\times n \cdot \tau_{\alpha}$, where $n$ and $\tau_{\alpha}$, $\alpha \in \{1,2\}$, denote the unit normal and tangent vectors and $\mathbb{D}(v)$ is a dilatation tensor equal to $\frac{1}{2}\left(\nabla v + \nabla^{\perp} v\right)$. The function $\kappa$ represents the curvature of $S$. From the physical point of view it may be interpreted as tangential ``slip'' velocity being proportional to tangential stress with a factor of proportionality depending only on the curvature (see e.g. \cite{kell}, \cite{Clopeau:1998vj}).

Nevertheless, under Dirichlet condition it is still possible to prove the existence of regular solution by adopting different technique (see e.g. \cite{bol}). From mathematical perspective, the boundary conditions which shall complement problem \eqref{p1} have to provide the energy estimates. Although not for every available choice of these boundary conditions the existence of regular solutions has been proved, but it does not mean that such proofs will never appear.

As for the initial data we simply put
\begin{equation}\label{p7}
	v\vert_{t = t_0} = v(t_0), \qquad \omega\vert_{t = t_0} = \omega(t_0) \quad \text{in $\Omega$}.
\end{equation} 

\section{Notation}\label{sec2.1}

Before we formulate the main result let us shortly clarify the notation deployed throughout this work.

The most frequently used notation in the sequel will be $\Omega^t$, which denotes the product $\Omega\times(t_0,t)$. Unless stated directly, we only assume that $0 \leq t_0 < t < \infty$.

By $c$ we denote a generic constant that may change from line to line. Additionally, such constants are subscripted with appropriate symbols, which indicate the dependence on the domain, embedding theorems, etc. The possible values are listed below:
	\begin{description}
	\item[$c_{\alpha,\beta,\nu,\nu_r}$] appears when the constant $c$ depends on the viscosity coefficients,
	\item[$c_{I}$] refers to embedding theorems (e.g. $H^1(\Omega) \hookrightarrow L_6(\Omega)$), 
	\item[$c_P$] refers to the Poincar\'e inequality,
	\item[$c_{\Omega}$] indicates the direct dependence on the geometry of the domain $\Omega$.
\end{description} 
Our motivation to keep the information which factors contribute to the constants is caused by the necessity of precise control of their dependence with respect to time. Time dependent constants would surely have a negative impact on the proof of global in time solutions. We do not say that such proof would not be possible, but unquestionably much harder. Note also, that since the Poincar\'e constant and the embedding constant depend on the domain, we could write $\Omega$ instead of $P$ and $I$ every time they appear. We decided not to make such generalization in order to keep the passage from line to line readable and clear. 

Throughout this work we shall use the following notation to simplify the formulas:
\begin{align*}
	h &= v_{,x_3}, & & & \theta &= \omega_{,x_3}.
\end{align*}

In energy estimates the initial and the external data in $L_p$-norms will appear. Therefore we introduce the following quantities to shorten formulas:

\begin{equation}\label{eq14}
	\begin{aligned}
		E_{v,\omega}(t) &:= \norm{f}_{L_2(t_0,t;L_{\frac{6}{5}}(\Omega))} + \norm{g}_{L_2(t_0,t;L_{\frac{6}{5}}(\Omega))} + \norm{v(t_0)}_{L_2(\Omega)} + \norm{\omega(t_0)}_{L_2(\Omega)}, \\
		E_{h,\theta}(t) &:= \norm{f_{,x_3}}_{L_2(t_0,t;L_{\frac{6}{5}}(\Omega))} + \norm{g_{,x_3}}_{L_2(t_0,t;L_{\frac{6}{5}}(\Omega))} + \norm{h(t_0)}_{L_2(\Omega)} + \norm{\theta(t_0)}_{L_2(\Omega)}.
	\end{aligned}
\end{equation}
The following function will be of particular interest
\begin{equation}\label{eq26}
	\delta(t) := \norm{f_{,x_3}}^2_{L_2(\Omega^t)} + \norm{g_{,x_3}}^2_{L_2(\Omega^t)} + \norm{\Rot h(t_0)}^2_{L_2(\Omega)} + \norm{h(t_0)}^2_{L_2(\Omega)} + \norm{\theta(t_0)}^2_{L_2(\Omega)}.
\end{equation}
This is the most important quantity since it expresses the smallness assumption which has to be made in order to prove the existence of global and regular solutions. It contains no $L_2$-norms of the initial or the external data but only $L_2$-norms of their derivatives alongside the axis of the pipe. In other words the data need not to be small but small must be their rate of change. 

Also note, that if we considered Navier-Stokes equations only and the external data were missing, $h(t_0)$ had a gradient structure (i.e. $h(t_0) = \nabla A$, $A\colon \mathbb{R}^3 \to \mathbb{R}$; problem for $h$ is demonstrated in Lemma \ref{p5}), then $\delta(t)$ would be equal to zero. Although $h(t_0)$ in $L_2$-norm appears it can be estimated by $\Rot h(t_0)$ (see Lemma \ref{l2}). The same would also hold for $f_{,x_3}$ with a gradient structure. 

From time to time we use the rotation of vector field. By $\Rot F$, where $F\colon \mathbb{R}^3 \to \mathbb{R}^3$, we denote
\begin{equation*}
	\Rot F = [F_{3,x_2} - F_{2,x_3}, F_{1,x_3} - F_{3,x_1}, F_{2,x_1} - F_{1,x_2}].
\end{equation*}

To define certain functions spaces that will be used frequently in the sequel we simply follow \cite[Ch. 3, \S 1.1]{luk1}, \cite[Ch. 2, \S 3]{lad} and \cite[Ch. 1, \S 1.1]{tem}: 

\begin{enumerate}
	\item[$\bullet$] $L_p(\Omega)$ is the set of all Lebesgue measurable function $u\colon \Omega \to \mathbb{R}^n$ with the norm
		\begin{equation*}
			\norm{u}_{L_p(\Omega)} = \left(\int_{\Omega} \abs{u}^p\, \ud x\right)^{\frac{1}{p}},
		\end{equation*}
	\item[$\bullet$] $W^m_p(\Omega)$, where $m \in \mathbb{N}$, $p \geq 1$, is the closure of $\mathcal{C}^{\infty}(\Omega)$ in the norm
		\begin{equation*}
			\norm{u}_{W^m_p(\Omega)} = \left(\sum_{\abs{\alpha} \leq m} \norm{\Ud^{\alpha} u}_{L_p(\Omega)}^p\right)^{\frac{1}{p}},
		\end{equation*}
	\item[$\bullet$] $H^k(\Omega)$, where $k \in \mathbb{N}$, is simply $W^k_2(\Omega)$, 
	\item[$\bullet$] $W^{2,1}_p(\Omega^t)$, where $p \geq 1$, is the closure of $\mathcal{C}^{\infty}(\Omega\times(t_0,t_1))$ in the norm
		\begin{equation*}
			\norm{u}_{W^{2,1}_p(\Omega^t)} = \left(\int_{t_0}^{t_1}\!\!\!\int_{\Omega} \abs{u_{,xx}(x,s)}^p + \abs{u_{,x}(x,s)}^p + \abs{u(x,s)}^p + \abs{u_{,t}(x,s)}^p\, \ud x\, \ud s\right)^{\frac{1}{p}},
		\end{equation*}
	\item[$\bullet$] $H^1_0(\Omega)$ is the closure of $\mathcal{C}^{\infty}_0(\Omega)$ in the norm
		\begin{equation*}
			\norm{u}_{H^1_0(\Omega)} = \left(\int_{\Omega} \abs{\nabla u(x)}^2\, \ud x\right)^{\frac{1}{2}},
		\end{equation*}
	\item[$\bullet$] $L_q(t_0,t_1;X)$, where $q \geq 1$ and $X$ is a Banach space, is the set of all strongly measurable functions defined on the interval $[t_0,t_1]$ with values in $X$ with finite norm defined by
		\begin{equation*}
			\norm{u}_{L_q(t_0,t_1;X)} = \left(\int_{t_0}^{t_1}\norm{u(s)}_X^q\, \ud s\right)^{\frac{1}{q}},
		\end{equation*}
		where $1 \leq p < \infty$ and by
		\begin{equation*}
			\norm{u}_{L_{\infty}(t_0,t_1;X)} = \underset{t_0\leq s \leq t_1}{\esssup}\norm{u(s)}_X,
		\end{equation*}
		for $q = \infty$,
	\item[$\bullet$] $V^k_2(\Omega^t)$, where $k \in \mathbb{N}$, is the closure of $\mathcal{C}^{\infty}(\Omega\times(t_0,t_1))$ in the norm
		\begin{equation*}
			\norm{u}_{V^k_2(\Omega^t)} = \underset{t\in (t_0,t_1)}{\esssup}\norm{u}_{H^k(\Omega)} \\
    +\left(\int_{t_0}^{t_1}\norm{\nabla u}^2_{H^{k}(\Omega)}\, \ud t\right)^{1/2}.
		\end{equation*}
\end{enumerate}

\section{Weak and strong solutions}\label{sec3.0}

\begin{defi}\label{defi1}
	Let $v(t_0) \in L_2(\Omega)$ and $\omega(t_0) \in L_2(\Omega)$. By a weak solution to problem \eqref{p1} complemented with the boundary conditions \eqref{p2} we mean a pair of functions $(v,\omega)$ such that $v,\omega\in V_2^0(\Omega^{t_1})$ (for definition of $V_2^0(\Omega^t)$ see Section \ref{sec2.1}) and $\Div v = 0$, $v\cdot n\vert_S = 0$, which satisfies the integral identities
	\begin{subequations}
	\begin{multline}\label{eq34}
		\int_{\Omega^{t_1}} \big( -v\cdot \varphi_{,t} + (\nu + \nu_r)\Rot v \cdot \Rot \varphi + (v\cdot\nabla) v\cdot\varphi\big)\, \ud x\, \ud t \\
		+ \int_{\Omega} v\cdot\varphi\vert_{t = t_1}\, \ud x - \int_{\Omega}v\cdot\varphi\vert_{t = t_0}\, \ud x 
		= \int_{\Omega^{t_1}} \big(2\nu_r\Rot \omega\cdot \varphi + f\cdot \varphi\big)\, \ud x\, \ud t
	\end{multline}
	for any $\varphi \in H^1(\Omega^{t_1})$ such that $\Div \varphi = 0$, $\varphi \cdot n = 0$ and
	\begin{multline}\label{eq35}
		\int_{\Omega^{t_1}} \big( - \omega\cdot\psi_{,t} + \alpha\nabla \omega\cdot \nabla \psi + \beta \Div \omega\Div \psi + (v\cdot \nabla)\omega \cdot \psi + 4\nu_r \omega\cdot\psi\big)\, \ud x\, \ud t \\
		+ \int_{\Omega} \omega\cdot\psi\vert_{t = t_1}\, \ud x - \int_{\Omega} \omega\cdot \psi\vert_{t = t_0}\, \ud x 
		= \int_{\Omega^{t_1}} \big(2\nu_r\Rot v\cdot \psi + g\cdot \psi\big)\, \ud x\, \ud t
	\end{multline}
	\end{subequations}
	for any $\psi \in H^1(\Omega^{t_1})$ such that $\psi\vert_{S_1} = 0$ and $\psi'\vert_{S_2} = 0$, $\psi_{3,x_3}\vert_{S_2} = 0$.
\end{defi}

The existence of weak solutions is assured by the following result:
\begin{lem}\label{lem20}
	There exist at least one weak solution to problem \eqref{p1} in the sense of the above definition. 
\end{lem}

\begin{proof}
	The proof is quite standard and is based on Galerkin approximations, a priori estimates (see Lemma \ref{l4}) and compactness method. It can be found in \cite{roj1} for general case or in \cite[Ch. 4, Theorem 1.6.1]{luk1}, \cite[Theorem 2.1]{luk2}.
\end{proof}

For the strong (or regular) solutions there are various definitions, which are equivalent. For our purposes we will use the following
\begin{defi}
	By a strong (or regular) solution to problem \eqref{p1} we mean a pair of functions $(v,\omega)$ such that $(v,\omega)\in V_2^1(\Omega^{t_1})\times V_2^1(\Omega^{t_1})$ (for the definition of $V_2^1(\Omega^t)$ see Section \ref{sec2.1}) satisfying \eqref{eq34} and \eqref{eq35}.
\end{defi}

\section{Main result}

Our main result is to demonstrate that under certain conditions there exists a unique solution to problem \eqref{p1} for any $t$ such that $0 \leq t \leq T< \infty$. 

\begin{thm}[large time existence]\label{t1}
	Let $E_{v,\omega}(t) < \infty$, $E_{h,\theta}(t) < \infty$. Suppose that $v(t_0), \omega(t_0) \in H^1(\Omega)$, $f, g \in L_2(\Omega^t)$. Finally, assume that $f_3\vert_{S_2} = 0$, $g'\vert_{S_2} = 0$. Then, for $\delta(t)$ sufficiently small there exists a unique solution $(v,\omega) \in W^{2,1}_2(\Omega^t)\times W^{2,1}_2(\Omega^t)$ to problem \eqref{p1} supplemented with the boundary conditions \eqref{p2} such that
	\begin{equation*}
			\norm{v}_{W^{2,1}_2(\Omega^t)} + \norm{\nabla p}_{L_2(\Omega^t)} \leq c_{\alpha,\nu,\nu_r,I,P,\Omega}\Big(E_{v,\omega}(t) + E_{h,\theta}(t) + \norm{f'}_{L_2(\Omega^t)} + \norm{v(t_0)}_{H^1(\Omega)} + 1\Big)^3
	\end{equation*}
	and
	\begin{multline*}
			\norm{\omega}_{W^{2,1}_2(\Omega^t)} \leq c_{\alpha,\nu,\nu_r,I,P,\Omega}\Big(E_{v,\omega}(t) + E_{h,\theta}(t) + \norm{f'}_{L_2(\Omega^t)} + \norm{g}_{L_2(\Omega^t)}\\
			+ \norm{v(t_0)}_{H^1(\Omega)} + \norm{\omega(t_0)}_{H^1(\Omega)} + 1\Big)^3.
	\end{multline*}
\end{thm}

Let us briefly outline how we prove the above theorem: The beginning point is the energy estimate for solutions to problem \eqref{p1} on the time interval $[t_0,t_1]$. It ensures the existence of weak solutions (\cite[Ch. 3, Theorem 1.6.1]{luk1}, \cite[Theorem 2.1]{luk2} or in general case: \cite{roj1}). In the next step we introduce several auxiliary problems which provide us with better estimates for $v$ and $\omega$ but for the price of some smallness assumption on the rate of the change of the data (see \eqref{eq26}). Finally, the application of regu1larity results for the Stokes system (see \cite{ala}) and general parabolic systems (see \cite{sol2}) leads to considerable improvement in the regularity of weak solutions. Next we utilize the Leray-Schauder fixed point theorem (see Lemma \ref{lem24}) to prove the existence of strong solution. Finally, we demonstrate its uniqueness.

\section{State of the art}

Since the 70' dozens of results concerning the existence of weak or strong solutions, various conditional regularity criteria or qualitative properties of solutions, also for a generalized system, the so-called magneto-micropolar, are known. For a comprehensive summary we refer the reader to \cite[Ch. 3, \S 5]{luk1}. However, let us shortly outline these results which are closely related to our work. 

One of the earliest results was established by Galdi and Rionero in\cite{gald} where they stated the the boundary value problem with Dirichlet boundary conditions for micropolar flows belongs to the same class of evolution problems as Navier-Stokes equations. Therefore they were also able to formulate existence results for the micropolar equations which are similar to those obtained for the Navier-Stokes. The direct proofs of the existence and uniqueness of (global) strong solutions to \eqref{p1} under the zero Dirichlet boundary conditions came later and were obtained by \L ukaszewicz in \cite{luk11}. He needed sufficiently large $\nu$ and small data in comparison to $\nu$. 

In \cite{ort} Ortega-Torres and Rojas-Medar showed that under certain smallness assumption on the initial and external data there exists a unique, global and strong solution to problem \eqref{p1} (in fact they considered the magneto-microhydrodynamic equations, but we can safely put the magnetic field $b$ to be equal to zero which yields problem \eqref{p1}). In contrast to \cite{luk11} they no longer assumed a decay for the external data as times goes to the inifinity.

A couple of years later the same authors proved (see \cite{RojasMedar:2005cv}) again the existence and uniqueness of strong solutions to \eqref{p1} complemented with zero Dirichlet boundary conditions both for $v$ and $\omega$. Their new proof used an interactive approach and required certain smallness of $L_2$-norms of the external data and absence of the initial data. 

Subsequently, Yamaguchi (see \cite{yam}) proved the existence of global and strong solutions to \eqref{p1} in bounded domains under zero Dirichlet boundary conditions. His proof is based on the semi-group approach and requires some smallness of the data. 

By application of iterative scheme, Boldrini, Dur\'an and Rojas-Medar proved (see \cite{bol}) the existence of local strong solutions $(v,\omega) \in W^{2,1}_p(\Omega^t)\times W^{2,1}_p(\Omega^t)$, $p > 3$, to problem \eqref{p1} complemented with zero Dirichlet boundary conditions for both $v$ and $\omega$ in bounded and unbounded domains in $\mathbb{R}^3$ with compact $\mathcal{C}^2$-boundaries.

To the best of our knowledge there are no results which are close to $2d$ solutions to \eqref{p1} which would justify the choice of domains of cylindrical type. It is also clear that most authors considers only the homogeneous Dirichlet boundary condition both for $v$ and $\omega$. Therefore it makes us a case for this detailed study. 

What needs to be particularly emphasized is the fact that in our work we do not assume any smallness on the initial or external data. Their $L_2$-norms can be large. However, the data cannot change significantly alongside the cylinder. Their derivatives in the $x_3$-directions must remain small all the time. Therefore the flow is somehow close to the constant with respect to one variable.

Note also that cylindrical domains or boundary conditions were considered for the standard Navier-Stokes equations (for a comprehensive summary we refer the reader to the Introduction in \cite{wm6}). 

\section{Auxiliary results}

\begin{lem}[Embedding theorem]\label{lem10}
	Let $\Omega$ satisfy the cone condition and let $q \geq p$. Set
	\begin{equation*}
		\kappa = 2 - 2r - s - 5\left(\frac{1}{p} - \frac{1}{q}\right) \geq 0.
	\end{equation*}
Then for any function $u \in W^{2,1}_{p}(\Omega^t)$ the inequality
	\begin{equation*}
		\norm{\partial^r_t \Ud_x^s u}_{L_q(\Omega^t)} \leq c_1 \epsilon^{\kappa}\norm{u}_{W^{2,1}_p(\Omega^t)} + c_2 \epsilon^{-\kappa + 2s - 2}\norm{u}_{L_p(\Omega^t)}
	\end{equation*}
	holds, where the constants $c_1$ and $c_2$ depend only on $p$, $q$, $r$, $s$ and $\Omega$ but do not depend on $t$.
\end{lem}
For the proof of the lemma we refer the reader to \cite[Ch.2, \S 3, Lemma 3.3]{lad}. As before, we lay emphasis on the fact that the constant $c_1$ and $c_2$ do not depend on time.

\begin{lem}\label{lem11}
	Suppose that $u \in V_2^0(\Omega^t)$. Then $u \in L_q(0,t;L_p(\Omega))$ and
	\begin{equation*}
		\norm{u}_{L_q(t_0,t;L_p(\Omega))} \leq c_{p,I} \norm{u}_{V^0_2(\Omega^t)}
	\end{equation*}
	holds under the condition $\frac{3}{p} + \frac{2}{q} = \frac{3}{2}$, $2 \leq p \leq 6$.
\end{lem}
Let us emphasize that the constant that appears on the right-hand side does not depend on time.
\begin{proof}
	We want to show that
	\begin{equation*}
		\norm{u}_{L_q(t_0,t;L_p(\Omega))} \leq c_{p,q,I} \left(\norm{u}_{L_2(t_0,t;H^1(\Omega))} + \norm{u}_{L_\infty(t_0,t;L_2(\Omega))}\right).
	\end{equation*}
	For $L_p$-spaces we have the interpolation inequality
	\begin{equation*}
		\norm{u}_{L_p(\Omega)} \leq \norm{u}_{L_r(\Omega)}^{\theta}\norm{u}_{L_s(\Omega)}^{1 - \theta},
	\end{equation*}
	where $s < p < r$ and $\frac{1}{p} = \frac{\theta}{r} + \frac{1 - \theta}{s}$. Integrating with respect to time yields
	\begin{multline*}
		\norm{u}_{L_q(t_0,t_1;L_p(\Omega))} = \left(\int_{t_0}^{t} \norm{u(\tau)}_{L_p(\Omega)}^q\, \ud \tau\right)^{\frac{1}{q}} \leq \left(\int_{t_0}^{t} \norm{u(\tau)}_{L_r(\Omega)}^{q\theta}\norm{u(\tau)}_{L_s(\Omega)}^{q(1 - \theta)}\, \ud \tau\right)^{\frac{1}{q}} \\
		\leq \sup_{\tau \in (t_0,t_1)} \norm{u(\tau)}_{L_r(\Omega)}^{\theta}\left(\int_{t_0}^{t} \norm{u(\tau)}_{L_s(\Omega)}^{q(1 - \theta)}\, \ud \tau\right)^{\frac{1}{q}}.
	\end{multline*}
	Using the Young inequality gives
	\begin{equation*}
		\norm{u}_{L_q(t_0,t;L_p(\Omega))} \leq \theta \sup_{\tau \in (t_0,t)} \norm{u(\tau)}_{L_r(\Omega)} + (1 - \theta)\left(\int_{t_0}^{t} \norm{u(\tau)}_{L_s(\Omega)}^{q(1 - \theta)}\, \ud \tau\right)^{\frac{1}{q(1 - \theta)}}.
	\end{equation*}
	Next we set $r = 2$, $s = 6$ and $q(1 - \theta) = 2$. Then $\frac{1}{p} = \frac{\theta}{2} + \frac{1 - \theta}{6} = \frac{1 + 2\theta}{6}$ and finally
	\begin{equation*}
		\frac{3}{p} + \frac{2}{q} = \frac{1 + 2\theta}{2} + (1 - \theta) = \frac{1 + 2\theta + 2 - 2\theta}{2} = \frac{3}{2}.
	\end{equation*}

	For $p = 2$, $q = \infty$ or $p = 6$, $q = 2$ the estimate follows immediately from definition of the space $V_2^0(\Omega^t)$. This ends the proof.
\end{proof}

\begin{lem}[Leray-Schauder fixed point principle]\label{lem24}
	Let $X$ denote a Banach space. Suppose that $\Phi\colon X\times[0,1]\to X$ satisfies the following conditions:
		\begin{itemize}
			\item for any fixed $\lambda \in [0,1]$ the mapping $\Phi(\cdot,\lambda)\colon X \to X$ is continuous,
			\item for any fixed $x \in X$ the mapping $\Phi(x,\cdot)\colon[0,1] \to X$ is uniformly continuous,
			\item there exists a bounded subset $A \subset X$ such that every fixed point of the mapping $\Phi(\cdot,\lambda)\colon X \to X$ for any choice of $\lambda \in [0,1]$ belongs to $A$,
			\item the mapping $\Phi(\cdot,0)$ has only one fixed point. 
		\end{itemize}
		Then, $\Phi(\cdot,1)$ has at least one fixed point.
\end{lem}

\begin{proof}
	For the proof we refer the reader to \cite{daf}.
\end{proof}

In further considerations we will often integrate by parts. To avoid repetition of some calculations, we shall demonstrate the most general case in the below Lemma:

\begin{lem}[On integration by parts]\label{l1}
	Let $u$ and $w$ belong to $H^1(\Omega)$. Then
	\begin{align*}
		\int_{\Omega} \Rot u \cdot w\, \ud x &= \int_{\Omega} \Rot w \cdot u\, \ud x + \int_S u \times n \cdot w\, \ud S \\
		&= \int_{\Omega} \Rot w \cdot u\, \ud x - \int_S w \times n \cdot u\, \ud S
	\end{align*}
\end{lem}
\begin{proof}
	It is an easy computation.
\end{proof}

\begin{lem}\label{lem32}
	Let $\kappa$ and $\mathbb{D}(v)$ denote the curvature of the boundary $S$ and dilatation tensor, i.e. $\mathbb{D}(v) = \frac{1}{2}\left(\nabla v + \nabla^{\perp} v\right)$, respectively. Then
	\begin{equation*}
		n\cdot\mathbb{D}(v)\cdot\tau_{\alpha} + \kappa (2 - \alpha)v\cdot\tau_1 = -\frac{1}{2}\Rot v\times n \cdot \tau_{\alpha}
	\end{equation*}
	holds on $S$, where $\tau_{\alpha}$, $\alpha \in \{1,2\}$ is a tangent vector.
\end{lem}

\begin{proof}
	We have
	\begin{equation*}
		n\cdot \mathbb{D}(v)\cdot \tau_\alpha = \frac{1}{2}\sum_{i,j} n_i \left(\frac{\partial v_j}{\partial x_i} + \frac{\partial v_i}{\partial x_j}\right)\tau_{\alpha j} = \frac{1}{2}\sum_{i,j} n_i \left(\frac{\partial v_j}{\partial x_i} - \frac{\partial v_i}{\partial x_j}\right)\tau_{\alpha j} + \sum_{i,j} n_i \frac{\partial v_i}{\partial x_j}\tau_{\alpha j}=: I_1 + I_2.
	\end{equation*}
	Clearly
	\begin{equation*}
		I_1 = -\frac{1}{2}\Rot v\times n \cdot \tau_\alpha.
	\end{equation*}
	For $I_2$ we have
	\begin{equation*}
		I_2 = \sum_{i,j}\left(\frac{\partial\big(n_iv_i\big)}{\partial x_j} \tau_{\alpha j} - v_i\frac{\partial n_i}{\partial x_j}\tau_{\alpha j}\right) = \sum_{i,j}v_i\frac{\partial n_i}{\partial x_j}\tau_{\alpha j}.
	\end{equation*}
	Since $n\vert_{S_2} = [0,0,1]$, we obtain $I_2\vert_{S_2} = 0$. On $S_1$ we use \eqref{p4}. We see that $n$ does not depend on $x_3$, which implies that $i,j\in\{1,2\}$ and $\alpha = 1$. By simple computation we get
	\begin{equation*}
		\sum_{i,j} \frac{\partial n_i}{\partial x_j}\tau_{1j} = \frac{1}{\abs{\nabla \varphi}^2}\big(-\varphi_{,x_1x_1}\varphi_{,x_2} + \varphi_{,x_1x_2}\varphi_{,x_1}, -\varphi_{,x_1x_2}\varphi_{,x_2} + \varphi_{,x_2x_2}\varphi_{,x_1}, 0\big). 
	\end{equation*}
	We can express $v$ in the basis $n,\tau_1,\tau_2$ as follows: 
	\begin{equation*}
		v = (v\cdot n)n + (v\cdot \tau_1)\tau_1 + (v\cdot \tau_2)\tau_2.
	\end{equation*}
	Since $v\cdot n\vert_S = 0$, we get 
	\begin{equation*}
		I_2 =  \frac{(v\cdot \tau_1)}{\abs{\nabla \varphi}^3} \big(\varphi_{,x_1x_1}\varphi^2_{,x_2} - 2\varphi_{,x_1x_2}\varphi_{,x_1}\varphi_{,x_2}  + \varphi_{,x_2x_2}\varphi^2_{,x_1}\big).
	\end{equation*}
	For a curve given in an implicit form, i.e. by $\varphi(x_1,x_2) = c$, the curvature is defined as
	\begin{equation*}
		\kappa = \frac{\begin{vmatrix}\varphi_{,x_1x_1} & \varphi_{,x_1x_2} & \varphi_{,x_1} \\ \varphi_{,x_2x_1} & \varphi_{,x_2x_2} & \varphi_{,x_2} \\ \varphi_{,x_1} & \varphi_{,x_2} & 0\end{vmatrix}}{\abs{\nabla \varphi}^3}.
	\end{equation*}
	Thus,
	\begin{equation*}
		I_2 = -\kappa(v\cdot \tau_1),
	\end{equation*}
	which ends the proof.	
\end{proof}

\begin{lem}[Boundary conditions on $S_2$]\label{lem3}
	Let \eqref{p2} be satisfied. Then
	\begin{equation*}
		\begin{aligned}
			&(a) & &v_3\vert_{S_2} = 0, \quad v_{3,x_1}\vert_{S_2} = 0, \quad v_{3,x_2}\vert_{S_2} = 0, \\
			&(b) & &(\Rot v)'\vert_{S_2} = 0, \quad \chi_{,x_3}\vert_{S_2} = 0, \\
			&(c) & &h'\vert_{S_2} = 0, \quad h_{3,x_3}\vert_{S_2} = 0,
    \end{aligned}
	\end{equation*}
	where $(\Rot v)' = \big((\Rot v)_1,(\Rot v)_2,0\big)$, $h' = (h_1,h_2,0)$.
\end{lem}
\begin{proof}
	As immediate consequence of \eqref{p2}$_{1,2}$ we get
	\begin{equation*}
		v\cdot n\vert_{S_2} = v_3\vert_{S_2} = 0 \qquad \Rightarrow \qquad v_{3,x_1}\vert_{S_2} = 0 \quad \textrm{and} \quad v_{3,x_2}\vert_{S_2} = 0
	\end{equation*}
	and
	\begin{equation*}
		\Rot v \times n\vert_{S_2} = 0 \Leftrightarrow \left. \begin{array}{ll} v_{3,x_2} - v_{2,x_3} = (\Rot v)_1\vert_{S_2} = 0 \\ v_{1,x_3} - v_{3,x_1} = (\Rot v)_2\vert_{S_2} = 0 \end{array} \right\} \Rightarrow v_{1,x_3}\vert_{S_2} = v_{2,x_3}\vert_{S_2} = 0.
	\end{equation*}
	This yields
	\begin{equation*}
		\chi_{,x_3}\vert_{S_2} = v_{2,x_1x_3} - v_{1,x_2x_3}\vert_{S_2} = 0.
	\end{equation*}
	and from $\Div h = 0$ we have
	\begin{equation*}
		v_{3,x_3x_3}\vert_{S_2} = -v_{1,x_3x_1} - v_{2,x_3x_2}\vert_{S_2} = 0 \qquad \Rightarrow \qquad h_{3,x_3}\vert_{S_2} = 0.
	\end{equation*}
	This ends the proof.
\end{proof}

\begin{lem}\label{l2}
	Suppose that
		\begin{equation*}
			\begin{aligned}
				&\Rot u = \alpha & &\text{in $\Omega$}, \\
				&\Div u = \beta & &\text{in $\Omega$}
		\end{aligned}
	\end{equation*}
	with either $u \cdot n\vert_S = 0$ or $u \times n\vert_S = 0$.
	Then
	\begin{equation*}
		\norm{u}_{H^{k + 1}(\Omega)} \leq c_{\Omega}\big(\norm{\alpha}_{H^k(\Omega)} + \norm{\beta}_{H^k(\Omega)}).
	\end{equation*}
\end{lem}
\begin{proof}
	For the proof we refer the reader to \cite[Ch. 7, Thm. 6.1]{duv} (case $k = 0$) and \cite{sol1} (case $k \in \mathbb{N}$). In the latter general overdetermined elliptic systems were examined. In particular, the case of tangent components of $u$ was considered. 
\end{proof}

\begin{lem}\label{lem27}
	Let $\alpha$ and $\beta$ be positive constants. Let us define the operator $L$ by the formula
	\begin{equation*}
		L = - \alpha \triangle - \beta\nabla \Div = \sum_{i,j=1}^3 a^{ij} \partial_{x_i}\partial_{x_j}.
	\end{equation*}
	Next, consider the problem
	\begin{equation*}
		\begin{aligned}
			&Lu = f & &\text{in $\Omega$}, \\
			&u = 0 & &\text{on $S_1$}, \\
			&u' = 0, \quad u_{3,x_3} = 0 & &\text{on $S_2$}.
		\end{aligned}
	\end{equation*}
	Then $L$ is uniformly elliptic and for any $f \in L_2(\Omega)$ the estimate
	\begin{equation*}
		\norm{u}_{H^2(\Omega)} \leq c_{\alpha,\beta,\Omega}\norm{f}_{L_2(\Omega)}
	\end{equation*}
	holds.
\end{lem}

\begin{proof}
	First we check the ellipticity. It means that there should exist a constant $\theta > 0$ such that
	\begin{equation*}
		\sum_{i,j = 1}^3 a^{ij} \xi_i\xi_j \geq \theta \abs{\xi}^2
	\end{equation*}
	for all $\xi \in \mathbb{R}^3$. In other words we want to demonstrate that the symmetric matrix $A = (a^{ij})$ of the form 
	\begin{equation*}
		\begin{bmatrix}
			\alpha + \beta & \beta & \beta \\
			\beta & \alpha + \beta & \beta \\
			\beta & \beta & \alpha + \beta
		\end{bmatrix}
	\end{equation*}
	which corresponds to the operator $L$, is positive definite and its smallest eigenvalue is greater than $\theta$. To compute its eigenvalues we solve the equation $(\det A - t\Id) = 0$ with respect to $t$. We see that
	\begin{equation*}
		\det A - t\Id = \left(\alpha + \beta - t\right)^3 + 2\beta^3 - 3\beta^2(\alpha + \beta) = (\alpha + \beta - t)^3 - \beta^3 - 3\alpha\beta^2 = f(t).
	\end{equation*}
	Since
	\begin{equation*}
		f'(t) = -3(\alpha + \beta - t)^2 < 0
	\end{equation*}
	we immediately deduce that $f$ is decreasing. Therefore there exists only one $t^*$ such that $f(t^*) = 0$. Since $f(0) = (\alpha + \beta)^3 - \beta^3 - 3\alpha\beta^2 = \alpha^3 + 3\alpha^2\beta > 0$ we infer that $t^* > 0$. This implies that the smallest eigenvalue of the matrix $A$ is positive. Hence the operator $L$ is uniformly elliptic.

	To prove the estimate we proceed in a standard way. First, we introduce a partition of unity $\sum_{k = 0}^N \zeta_k(x_3) = 1$ on $\Omega$. Let us denote $\bar u = u \zeta_k$. For fixed $k$ four cases may occur:
	\begin{description}
		\item[1. $\supp\zeta_k \cap S = \emptyset$.] In this case we deal with the problem in the whole space
			\begin{equation*}
				\begin{aligned}
					&L\bar u = \bar f + [0,0,2\alpha\nabla u_3\cdot \nabla \zeta_k + \alpha u_3\triangle \zeta_k] + \beta\zeta_{k,x_3}\nabla \omega_3 =: F_k & &\text{in $\supp\zeta_k\cap\Omega$}, \\
					&\bar u = 0 & &\text{on $\partial\left(\supp \zeta_k\cap \Omega\right)$}.
				\end{aligned}
			\end{equation*}
			From the classical theory (see \cite[Ch.2, \S 3.2, Thm. 3.1]{lions}) it follows that
			\begin{equation}\label{eq45}
				\norm{\bar u}_{H^2(\supp\zeta_k\cap\Omega)} \leq c_{\Omega}\left(\norm{F_k}_{L_2(\supp\zeta_k\cap\Omega)} + \norm{\bar u}_{H^1(\supp\zeta_k\cap\Omega)}\right). 
			\end{equation}
		\item[2. $\supp\zeta_k \cap S_1 \neq \emptyset$, $\supp\zeta_k \cap S_2 = \emptyset$.] Since $\bar u\vert_{S_1} = 0$ we obtain that $\bar u\vert_{\partial(\supp\zeta_k\cap\Omega)} = 0$. Next, we transform the set $\supp \zeta_k\cap \Omega$ into the half-space and apply the result from classical theory for the half-space (see \cite[Ch. 2, \S 4.5, Thm. 4.3]{lions}), which finally gives \eqref{eq45}. For the meticulous details we refer the reader to the proof of Theorem 5.1 in \cite[Ch. 2, \S 5.1]{lions}.
		\item[3. $\supp\zeta_k \cap S_1 \neq \emptyset$, $\supp\zeta_k \cap S_2 \neq \emptyset$.] Let us recall that $u'\vert_{S_2} = 0$ and $u_{3,x_3}\vert_{S_2} = 0$. Thus $\bar u'\vert_{S_2} = 0$ and $\bar u_{3,x_3}\vert_{S_2} = u_3\zeta_{k,x_3} = 0$, which follows from the fact that $\zeta_k = \zeta_k(x_3)$. It allows us to reflect the function $\bar u$ outside the cylinder according to the formula
			\begin{equation*}
				\check u (x) = \begin{cases} \bar u(x) & x_3 \in \overline{\supp\zeta_k\cap \Omega}, \\
													(\bar u'(\bar x),-\bar u_3(\bar x)) & x_3 \leq -a, \\
													(\bar u'(\tilde x), -\bar u_3(\tilde x)) & x_3 \geq a,\end{cases}
			\end{equation*}
			where $\bar x = (x',-2a - x_3)$ and $\tilde x = (x', 2a - x_3)$. Note that $\bar u\vert_{\partial(\supp\zeta_k\cap\Omega)} = 0$, so we may proceed as in Case 2.
		\item[4. $\supp\zeta_k \cap S_1 = \emptyset$, $\supp\zeta_k \cap S_2 \neq \emptyset$.] This case does not differ from the previous one in any major way. We also reflect the function $\bar u$ as described above and proceed as in Case 2. 
	\end{description}
	Summing over $k$ yields
	\begin{equation*}
		\norm{u}_{H^2(\Omega)} \leq \sum_{k = 1}^N \norm{u\zeta_k}_{H^2(\Omega)} \leq c_{\Omega}\sum_{k = 1}^N\left(\norm{F_k}_{L_2(\Omega)} + \norm{u\zeta_k}_{H^1(\Omega)}\right) \leq c_{\alpha,\beta,\Omega}\left(\norm{f}_{L_2(\Omega)} + \norm{u}_{H^1(\Omega)}\right).
	\end{equation*}
	It remains to eliminate the last term on the right-hand side. We shall prove that the inequality $\norm{Lu}_{L_2(\Omega)} = \norm{f}_{L_2(\Omega)} \geq c_{\Omega} \norm{u}_{H^1(\Omega)}$ holds. Conversely, suppose that it is not true. Then there would exist sequences $u_l$ in $H^1(\Omega)$ and $f_l$ in $L_2(\Omega)$ such that
	\begin{equation*}
		\begin{aligned}
			&Lu_l = f_l & &\text{in $\Omega$}, \\
			&u_l = 0 & &\text{on $S_1$}, \\
			&u_l' = 0, \quad u_{3,x_3} = 0 & &\text{on $S_2$}
		\end{aligned}
	\end{equation*}
	and $\norm{u_l}_{H^1(\Omega)} \geq l \norm{f_l}_{L_2(\Omega)}$. Let us define $v_l = \frac{u_l}{\norm{u_l}_{H^1(\Omega)}}$. Then 
	\begin{equation*}
		\norm{Lv_l}_{L_2(\Omega)} = \frac{\norm{Lu_l}_{L_2(\Omega)}}{\norm{u_l}_{H^1(\Omega)}} = \frac{\norm{f_l}_{L_2(\Omega)}}{\norm{u_l}_{H^1(\Omega)}} \leq \frac{1}{l}.
	\end{equation*}
	From the above inequality we see that $v_l$ is bounded in $H^2(\Omega)$. The Rellich-Kondrachov Compactness Theorem implies that there exist a subsequence $v_{l_k}$ which converges strongly in $H^1(\Omega)$ to some element $v$. Thus
	\begin{equation*}
		\begin{cases} v_{l_k} \rightharpoonup v & \text{in $H^2(\Omega)$}, \\
							v_{l_k} \to v & \text{in $H^1(\Omega)$}.
		\end{cases}
	\end{equation*}
	On the other hand we see that for every $\phi \in \mathcal{C}^{\infty}_c(\Omega)$
	\begin{equation*}
		\int_{\Omega} v \cdot \phi_{,x_ix_j}\, \ud x = \lim_{l_k \to \infty} \int_{\Omega} v^{l_k} \cdot \phi_{,x_ix_j}\, \ud x = \lim_{l_k \to \infty} \int_{\Omega} v^{l_k}_{,x_ix_j} \cdot \phi\, \ud x = 0,
	\end{equation*}
	whereas $\norm{v_l}_{H^1(\Omega)} = 1$, which is a contradiction.
\end{proof}

\begin{lem}\label{lem16}
	Let us denote $\Omega^t = \Omega\times (t_0,t)$. Consider the following initial-boundary value problem
	\begin{equation}\label{eq41}
		\begin{aligned}
			&u_{,t} - \alpha\triangle u - \beta \nabla \Div u = F & &\text{in } \Omega^t, \\
			&u = 0 & &\text{on } S_1^t, \\
			&u' = 0, \qquad u_{3,x_3} = 0 & &\text{on } S_2^t, \\
			&u\vert_{t = t_0} = u(t_0) & &\text{on } \Omega\times\{t = t_0\},
		\end{aligned}
	\end{equation}
	where $\alpha$ and $\beta$ are positive constants. Assume that $F \in L_p(\Omega^t)$. Then there exist a unique solution $u$ such that $u \in W^{2,1}_p(\Omega^t)$ and
	\begin{equation*}
		\norm{u}_{W^{2,1}_p(\Omega^t)} \leq c_{\Omega} \left(\norm{F}_{L_p(\Omega)} + \norm{u(t_0)}_{W^{2 - \frac{2}{p}}_p(\Omega)}\right).
	\end{equation*}	
\end{lem}

\begin{proof}
	Let us introduce a partition of unity $\sum_{k = 1}^N \zeta_k(x_3) = 1$ and denote $\bar u = u\zeta_k$. Then we can repeat the considerations from the proof of previous Lemma (Lemma \ref{lem27}). However, there is a slight difference:
	\begin{description}
		\item[1. $\supp\zeta_k \cap S = \emptyset$.] In this case we have
			\begin{equation*}
				\begin{aligned}
					u_{,t} + L\bar u = \bar F + [0,0,2\nabla u_3\cdot \nabla \zeta_k + u_3\triangle \zeta_k] =: F_k& & &\text{in $\supp\zeta_k\cap\Omega$}, \\
					\bar u = 0 & & &\text{on $\partial\left(\supp \zeta_k\cap \Omega\right)$},
				\end{aligned}
			\end{equation*}
			which is seen as the problem in the whole space. From \cite[Thm. 1.1]{sol2} it follows that
			\begin{equation}\label{eq48}
				\norm{\bar u}_{W^{2,1}_p(\supp\zeta_k\cap\Omega^t)} \leq c_{\Omega} \left(\norm{F_k}_{L_p(\supp\zeta_k\cap\Omega^t)} + \norm{\bar u(t_0)}_{W^{2 - \frac{2}{p}}_p(\supp\zeta_k\cap\Omega)}\right).
			\end{equation}
	\end{description}
	Originally, the constant which appears on the right-hand side may depend on time. However, since we have the energy estimates for solutions to \eqref{eq41}, we can utilize \cite{wah} to exclude the time dependence of the constant. 
	The remaining cases are reduced to the ones presented in Lemma \ref{lem27} analogously: we consider three cases when the supports of the cut-off functions $\zeta_k$ touch the boundary and reduce every case to the problem in the half-space. Then we utilize Theorem 5.5 from \cite{sol2} to obtain \eqref{eq48} but in the half space. Summing over $k$ yields
	\begin{equation*}
		\norm{u}_{W^{2,1}_p(\Omega^t)} \leq c_{\Omega}\left(\norm{F}_{L_p(\Omega^t)} + \norm{u}_{L_p(t_0,t_1;W^1_p(\Omega))} + \norm{u(t_0)}_{W^{2 - \frac{2}{p}}_p(\Omega)}\right).
	\end{equation*}
	To eliminate the second term on the right-hand side we use the energy estimate for solutions to \eqref{eq41}. This ends the proof.
\end{proof}

\begin{rem}
	In the above proof we omitted certain details related to the estimates in the half space. In subsequent considerations they will also be omitted. We refer the interested reader to \cite[\S 6, proofs of Theorems 1.1 and 1.2]{sol2}. 

	It is worth mentioning that an alternative approach was presented in \cite[\S 3]{bur}.
\end{rem}

\begin{lem}\label{lem17}
	Consider the Stokes problem
	\begin{equation}\label{eq42}
		\begin{aligned}
			&v_{,t} - (\nu + \nu_r) \triangle v + \nabla p = F & &\text{in $\Omega^t$}, \\
			&\Div v = 0 & &\text{in } \Omega^t, \\
			&v\cdot n = 0 & &\text{on } S^t, \\
			&\Rot v \times n = 0 & &\text{on } S^t, \\
			&v\vert_{t = t_0} = v(t_0) & &\text{on } \Omega\times\{t = t_0\}.
		\end{aligned}
	\end{equation}
	If $F \in L_p(\Omega)$ then there exist a solution to the above problem such that $v \in W^{2,1}_p(\Omega^t)$ and the estimate
	\begin{equation*}
		\norm{v}_{W^{2,1}_p(\Omega^t)} + \norm{\nabla p}_{L_p(\Omega^t)} \leq c_{\Omega}\left( \norm{F}_{L_p(\Omega)} + \norm{v(t_0)}_{W^{2 - \frac{2}{p}}_p(\Omega)}\right)
	\end{equation*}
	holds.
\end{lem}

\begin{proof}
	For the proof we refer the reader to \cite{ala}. However, few remarks are required. The cited article is concerned with domains whose boundaries belong to $W^{2 - \frac{1}{p}}_p$ which imposes a fundamental limitation on $p$ --- it cannot be smaller or equal to dimension $d$ of the space ($d = 3$ in our case). However, we would like to set $p = 2$. Since $p$ is strongly related to the regularity of the boundary, we should assume that $S \in \mathcal{C}^2$. Then $p > 1$ (for a closer correspondence between the regularity class of $S$ and $p$ see \cite[p. 217, the first paragraph]{ala}). Note, that at least the H\"older continuity of the first derivative of $u$ is necessary to justify the calculations that appear.

	The boundary of the domain under consideration in this thesis has corners, thus it does not belong to $\mathcal{C}^2$ class. However, we may locally reflect the interior of the cylinder with respect to the planes $x_3 = \pm a$ (similarly as in proof of Lemma \ref{lem27}), which leads to two different cases, depending whether we are near to $S_1$ or not. In the first case we would deal, after straightening $S_1$, with the model problem in the half space, as it was demonstrated in \cite[Sec. 3]{ala}. In the second case we would consider a model problem in the whole space. 
	
	The above presented approach has a major drawback. The estimates obtained contain constants that depend on time (see \cite[Sec. 4]{ala}). But we have the energy estimates for solutions to \eqref{eq42} and therefore the application of the result from \cite{wah} implies that the constants are time independent.

	The last comment we make concerns the boundary conditions. In this thesis we use the Navier condition for the velocity field, whereas in \cite{ala} some more general condition was adopted, i.e. 
	\begin{equation}\label{eq49}
		n\cdot \mathbb{D}(v)\cdot \tau_i + \gamma v\cdot \tau_i = b_i. 
	\end{equation}
	The vectors $n$ and $\tau_i$ denote the normal outward and the tangent vectors to the boundary. By $\mathbb{D}(v)$ a dilatation tensor (i.e. $\frac{1}{2}\left(\nabla v + \nabla^{\perp} v\right)$) is understood. The positive constant $\gamma$ represents the friction. Thus, \eqref{eq49} describes the perfect slip with friction. On the other hand we know (see Lemma \ref{lem32}) that if $\gamma$ denoted the curvature of the boundary then the left hand-side in \eqref{eq49} would be equal to $-\frac{1}{2}(\Rot v\times n)\cdot \tau_{\alpha}$, which in our case would have implied that $b_i = 0$. 

Since we assume that $S_i\in \mathcal{C}^2$, $i \in \{1,2\}$, the curvature is a smooth function. In the proof presented in \cite{ala} the constant $\gamma$ was set to $1$ without the loss of generality which in case of our boundary conditions corresponds to the domain of circular cross-section, whose radius is equal to $1$. But as long as $\gamma$ remains a smooth function the whole proof can repeated without any significant difficulties. 

\end{proof}

\section{Auxiliary problems}

\begin{lem}\label{lem12}
	Let $v$, $\theta$ and $f_{,x_3}$ be given. Then the pair $(h,q)$ is a solution to the problem
	\begin{equation}\label{p5}
		\begin{aligned}
			&h_{,t} - (\nu + \nu_r)\triangle h + \nabla q = -v \cdot \nabla h - h \cdot \nabla v + 2\nu_r\Rot \theta + f_{,x_3} & &\text{in $\Omega^t$}, \\
			&\Div h = 0 & &\text{in $\Omega^t$}, \\
			&\Rot h \times n = 0, \quad h\cdot n = 0 & &\textrm{on $S_1^t$}, \\
			&h' = 0, \qquad h_{3,x_3} = 0 & &\textrm{on $S_2^t$}, \\
			&h\vert_{t = t_0} = h(t_0) & &\textrm{in $\Omega$}.
		\end{aligned}
	\end{equation}
\end{lem}

\begin{proof}
	Equations \eqref{p5}$_{1,2}$ follow directly from \eqref{p1}$_{1,3}$ by differentiating with respect to $x_3$. In the same way we obtain the boundary condition \eqref{p5}$_3$ from \eqref{p2}$_2$ because $x_3$ is the tangent direction. The condition \eqref{p5}$_4$ was proved in Lemma \ref{lem3}$(c)$. The initial condition \eqref{p5}$_5$ follows from \eqref{p7}.
\end{proof}

\begin{rem}\label{rem1}
	For the function $h$ from lemma above the Poincar\'e inequality holds. Since $h'$ vanishes on $S_2$ we only need to check if the integral of $h_3$ over $\Omega$ equals zero. We have
	\begin{equation*}
		\int_{\Omega} h_3\, \ud x = \int_{\Omega} v_{3,x_3}\, \ud x = \int_{S_2(x_3 = -a)} v_3\, \ud x' - \int_{S_2(x_3 = a)} v_3\, \ud x' = 0.
	\end{equation*}
\end{rem}

\begin{lem}\label{lem4}
    Let $h$, $\omega$, $g'$ and $g_{,x_3}$ be given. Then the function $\theta$ is solution to the problem
    \begin{equation}\label{p6}
        \begin{aligned}
            &\theta_{,t} - \alpha \triangle \theta - \beta \nabla \Div \theta + 4\nu_r \theta = -h\cdot \nabla \omega - v \cdot \nabla \theta + 2 \nu_r \Rot h + g_{,x_3} & &\text{in $\Omega^t$}, \\
            &\theta = 0 & &\text{on $S_1^t$}, \\
            &\theta_3 = 0, \qquad \theta'_{,x_3} = -\frac{1}{\alpha}g' & &\text{on $S_2^t$}, \\
            &\theta\vert_{t = t_0} = \theta(t_0) & &\text{in $\Omega$}.
        \end{aligned}
    \end{equation}
\end{lem}

\begin{proof}
	Equation \eqref{p6}$_1$ follows directly from \eqref{p1}$_2$ by differentiating along $x_3$ direction. Analogously, we compute the boundary condition \eqref{p6}$_2$, because $x_3$ is the tangent direction. 

	To prove \eqref{p6}$_3$ we take two first components of \eqref{p1}$_2$ 
	\begin{equation*}
		\omega'_{,t} + v\cdot \nabla \omega' - \alpha \triangle' \omega' - \alpha \partial^2_{x_3x_3} \omega' - \beta\nabla'\Div \omega + 4\nu_r \omega' = 2\nu_r (\Rot v)' + g'
	\end{equation*}
	and project them onto $S_2$. From Lemma \ref{lem3} we deduce immediately that
	\begin{equation*}
		-\alpha \theta'_{,x_3}\vert_{S_2} = g'\vert_{S_2}.
	\end{equation*}

	The initial condition \eqref{p6}$_4$ follows from \eqref{p7}.
\end{proof}

\begin{rem}\label{rem2}
	Let us notice that for $\theta$ the Poincar\'e inequality holds. For $\theta_3$, which vanishes on $S_2$ it is obvious. For $\theta'$ we simply calculate the mean value:
	\begin{equation*}
		\int_{\Omega} \theta'\, \ud x = \int_{\Omega} \omega'_{,x_3}\, \ud x = \int_{S_2(x_3 = -a)}\omega'\, \ud x' - \int_{S_2(x_3 = a)}\omega'\, \ud x' = 0,
	\end{equation*}
	which follows from \eqref{p2}$_4$.
\end{rem}

\begin{lem}\label{lem7}
	Let $v$, $h$, $\omega$ and $F_3$ be given. Then the function $\chi$ is solution to the following set of equations:
	\begin{equation}\label{p9}
		\begin{aligned}
			&\begin{aligned} 
            &\chi_{,t} + v\cdot \nabla \chi - h_3\chi + h_2v_{3,x_1} - h_1v_{3,x_2} - (\nu + \nu_r)\triangle \chi \\
            &\mspace{100mu} = F_3 + 2\nu_r\left((\Rot \omega)_{2,x_1} - (\Rot\omega)_{1,x_2}\right)
         \end{aligned}& &\textrm{in } \Omega^t,\\
         &\chi = 0 & &\text{on $S_1^t$}, \\
         &\chi,_{x_3} = 0 & &\text{on $S_2^t$}, \\
         &\chi\vert_{t = t_0} = \chi(t_0) & &\text{in $\Omega$}.
		\end{aligned}
	\end{equation}
\end{lem}
\begin{proof}
	To deduce \eqref{p9}$_1$ we simply differentiate \eqref{p1}$_1$ with respect to $x_1$, subtract it from \eqref{p1}$_1$ differentiated with respect to $x_2$ and take the third component.

	To get \eqref{p9}$_2$ we multiply \eqref{p2}$_2$ by $\tau_1$ (see \eqref{p4}$_1$). It yields
	\begin{equation*}
		\left(\Rot v \times n \right) \cdot \tau_1 = 0 \qquad \Rightarrow \qquad \Rot v \cdot \tau_2 = 0 \quad \Leftrightarrow \quad v_{2,x_1} - v_{1,x_2} = 0.
	\end{equation*}

	The condition \eqref{p9}$_3$ was derived in Lemma \ref{lem3}$(b)$. The initial condition \eqref{p9}$_4$ follows from \eqref{p7} in the same manner as \eqref{p9}$_1$ from \eqref{p1}$_1$.
\end{proof}

\section{Energy estimates}\label{sec3.1}

The prime goal of this Section is to establish certain basic energy estimates, we formulate three lemmas. The first one presents an estimate for the velocity and microrotation fields (see Lemma \ref{l4}) in $V_2^0(\Omega^t)$ space. In the second we derive estimates for the functions $h$ and $\theta$ (see Lemma \ref{lem8}) and in the third for the function $\chi$ (see Lemma \ref{lem9}). Unlike for the functions $v$ and $\omega$, the inequalities for $h$, $\omega$ and $\chi$ in $V_2^0(\Omega^t)$ contain on the right-hand side a term which is a priori unknown and cannot be estimated by the data. This term, i.e. $\norm{h}_{L_\infty(t_0,t_1;L_3(\Omega))}$, will therefore appear every time we make use of these inequalities. It will be particularly visible in the subsequent Section, where we pay attention to higher order derivatives of the functions $v$, $h$ and $\omega$. Finally, it will be absorbed by the left-hand side but at the cost of a smallness assumption. As a consequence we shall get an estimate for $v$ and $\omega$ in $W^{2,1}_2(\Omega^t)$ space in terms of the data only. 

The below Lemma demonstrates the estimate for the functions $v$ and $\omega$.

\begin{lem}\label{l4}
	Let $E_{v,\omega}(t) < \infty$ hold (see \eqref{eq14}$_1$). Then for any $t_0 \leq t \leq t_1$ we have
	\begin{equation*}
		\norm{v}_{V_2^0(\Omega^t)} + \norm{\omega}_{V_2^0(\Omega^t)} \leq c_{\alpha,\nu,I,\Omega} E_{v,\omega}(t).
    \end{equation*}
\end{lem}

\begin{proof}
	First, recall that $-\alpha \triangle \omega = \alpha\Rot\Rot \omega - \alpha\nabla\Div\omega$. Thus, multiplying \eqref{p1}$_1$ by $v$, \eqref{p1}$_2$ by $\omega$, integrating over $\Omega$ and utilizing Lemma \ref{l1} yields
	\begin{multline*}
		\frac{1}{2}\Dt \norm{v}_{L_2(\Omega)}^2 + (\nu + \nu_r)\norm{\Rot v}_{L_2(\Omega)}^2 + (\nu + \nu_r)\int_S \Rot v\times n \cdot v\,\ud S = 2\nu_r\int_\Omega \Rot \omega \cdot v\, \ud x + \int_\Omega f\cdot v\, \ud x, \\
		\frac{1}{2}\Dt \norm{\omega}_{L_2(\Omega)}^2 + \alpha\norm{\Rot \omega}_{L_2(\Omega)}^2 - \alpha\int_S \omega\times n\cdot \Rot \omega\, \ud S + (\alpha + \beta)\norm{\Div \omega}_{L_2(\Omega)}^2 + 4\nu_r\int_\Omega \omega^2\, \ud x \\
		= 2\nu_r\int_\Omega \Rot v \cdot \omega\, \ud x + \int_\Omega g\cdot \omega\, \ud x.
	\end{multline*}
	Observe that the term with the pressure vanished due to $\Div v = 0$ and $v\cdot n\vert_S = 0$. Because of \eqref{p2}$_{2,3,4}$ the boundary integrals above are equal to zero. Adding both equalities and utilizing Lemma \ref{l1} again gives
	\begin{multline*}
		\frac{1}{2}\Dt \norm{v}_{L_2(\Omega)}^2 + \frac{1}{2}\Dt \norm{\omega}_{L_2(\Omega)}^2 + (\nu + \nu_r)\norm{\Rot v}_{L_2(\Omega)}^2 \\
		+ \alpha\norm{\Rot \omega}_{L_2(\Omega)}^2 + (\alpha + \beta)\norm{\Div \omega}_{L_2(\Omega)}^2 + 4\nu_r\int_\Omega \omega^2\, \ud x \\
		= 4\nu_r\int_\Omega \Rot v \cdot \omega\, \ud x + \int_S \omega\times n \cdot v\,\ud S + \int_\Omega f\cdot v\, \ud x + \int_\Omega g\cdot \omega\, \ud x.
	\end{multline*}
	By the means of the H\"older and the Young inequalities with $\epsilon$ we can estimate the right hand side in the following way
    \begin{multline*}
        4\nu_r\int_\Omega \Rot v \cdot \omega\, \ud x + \int_\Omega f\cdot v\, \ud x + \int_\Omega g\cdot \omega\, \ud x \\
        \leq 4\nu_r\epsilon_1 \norm{\Rot v}_{L_2(\Omega)}^2 + \frac{\nu_r}{\epsilon_1}\norm{\omega}_{L_2(\Omega)}^2 + \epsilon_2\norm{v}^2_{L_6(\Omega)} + \frac{1}{4\epsilon_2}\norm{f}^2_{L_{\frac{6}{5}}(\Omega)} \\
        + \epsilon_3\norm{\omega}^2_{L_6(\Omega)} + \frac{1}{4\epsilon_3}\norm{g}^2_{L_{\frac{6}{5}}(\Omega)}. 
    \end{multline*}
    Taking now $\epsilon_1 = \frac{1}{4}$, $\epsilon_2 = \frac{\nu c_\Omega}{2c_I}$, and $\epsilon_3 = \frac{\alpha c_\Omega}{2c_I}$, where the constant $c_I$ comes from the imbedding $H^1(\Omega) \hookrightarrow L_6(\Omega)$, and using Lemma \ref{l2} we see that
    \begin{equation*}
        \Dt \norm{v}_{L_2(\Omega)}^2 + \Dt\norm{\omega}_{L_2(\Omega)}^2 + \frac{\nu}{c_\Omega} \norm{v}_{H^1(\Omega)}^2 + \frac{\alpha}{c_{\Omega}} \norm{\omega}_{H^1(\Omega)}^2 \leq \frac{c_I}{\nu c_\Omega}\norm{f}^2_{L_{\frac{6}{5}}(\Omega)} + \frac{c_I}{\alpha c_\Omega}\norm{g}^2_{L_{\frac{6}{5}}(\Omega)}.
    \end{equation*}
    Integrating with respect to time on $(t_0,t)$ yields
    \begin{multline*}
        \norm{v}_{V_2^0(\Omega^t)}^2 + \norm{\omega}_{V_2^0(\Omega^t)}^2 \\
        \leq \frac{1}{\min\left\{1,\frac{\nu}{c_\Omega}, \frac{\alpha}{c_{\Omega}}\right\}} \bigg(\frac{c_I}{\nu c_\Omega}\norm{f}^2_{L_2(t_0,t;L_{\frac{6}{5}}(\Omega))} + \frac{c_I}{\alpha c_\Omega}\norm{g}^2_{L_2(t_0,t;L_{\frac{6}{5}}(\Omega))} \\
        + \norm{v(t_0)}_{L_2(\Omega)}^2 + \norm{\omega(t_0)}_{L_2(\Omega)}^2\bigg).
    \end{multline*}
    This completes the proof.
\end{proof}

In the second Lemma in this Section we are looking for an estimate for $h$ and $\theta$. As mentioned, we obtain only inequality with an a priori unknown term on the right-hand side, which is assumed to be finite. At this stage we are not provided with any appropriate means to control that term, thereby postponing its estimation till more conclusive results are derived (see Lemma \ref{lem6} in the next Section).

Before we formulate Lemma, let us state the following remarks:
\begin{rem}\label{rem3}
	Observe that $q\vert_{S_2} = f_3\vert_{S_2}$. 

	Indeed, taking the third component of \eqref{p1}, projecting it into $S_2$ and using Lemma \ref{lem3} gives
	\begin{equation*}
		q\vert_{S_2} = -v\cdot \nabla v \cdot n\vert_{S_2} + f_3\vert_{S_2} + 2\nu_r\Rot \omega\cdot n\vert_{S_2} = f_3\vert_{S_2}.
	\end{equation*}
\end{rem}

\begin{rem}\label{rem7}
	We will prove that
	\begin{equation*}
		\norm{h}_{H^{k + 1}(\Omega)} \leq c_{\Omega}\norm{\Rot h}_{H^k(\Omega)}.
	\end{equation*}
	To this end, let us first consider the following problem
	\begin{equation*}
		\begin{aligned}
			&\Rot h = \alpha & &\text{in $\Omega$}, \\
			&\Div h = 0 & &\text{in $\Omega$}, \\
			&h \cdot n = 0 & &\text{on $S_1$},\\
			&h \times n = 0 & &\text{on $S_2$}.
		\end{aligned}
	\end{equation*}
	Introduce a partition of unity $\sum_{k = 1}^N \zeta_k(x_3) = 1$. If we denote $\bar h = h \zeta_k$, then the above system becomes
	\begin{equation*}
		\begin{aligned}
			&\Rot \bar h = \bar \alpha + [-h_2\zeta_{k,x_3},h_1\zeta_{k,x_3},0] & &\text{in $\supp\zeta_k\cap\Omega$}, \\
			&\Div \bar h = h_3\zeta_{k,x_3} & &\text{in $\supp\zeta_k\cap\Omega$}, \\
			&\bar h \cdot n = 0 & &\text{on $\supp\zeta_k\cap S_1$}, \\
			&\bar h \times n = 0 & &\text{on $\supp\zeta_k\cap S_2$}.
		\end{aligned}
	\end{equation*}
	Next, we perform similar computations as we did in proof of Lemma \ref{lem27}.
\end{rem}

\begin{lem}\label{lem8}
	Let $E_{v,\omega}(t) < \infty$ and $E_{h,\theta}(t) < \infty$. Additionally, let $g'\vert_{S_2} = 0$, $f_3\vert_{S_2} = 0$. Finally, assume that $\norm{h}_{L_\infty(t_0,t;L_3(\Omega))} < \infty$. Then
	\begin{equation*}
		\norm{h}_{V_2^0(\Omega^t)} + \norm{\theta}_{V_2^0(\Omega^t)} \leq c_{\alpha,\nu,I,\Omega} \left(E_{v,\omega}(t) \norm{h}_{L_{\infty}(t_0,t;L_3(\Omega))} + E_{h,\theta}(t)\right).
	\end{equation*}
\end{lem}

Note, that the conditions $g'\vert_{S_2} = 0$ and $f_3\vert_{S_2} = 0$ can be dropped, but then we would have to deal with two non-trivial boundary integrals on $S_2$. We would be even capable to estimate them in suitable norms by application of the extension and the interpolation theorems, which would result in the appearance of the term $c_I\Big(\norm{h}^2_{L_2(\Omega^t)} + \norm{\theta}^2_{L_2(\Omega^t)}\Big)$ on the right-hand side. But we should not forget that the presented lemma serves only a supporting role in the proof of global existence of regular solutions to problem \eqref{p1}. Since we have no control over the magnitude of $c_I$ we would encounter insurmountable difficulties in further considerations (see the proof of Lemma \ref{lem6}). 

\begin{proof}
	Multiplying \eqref{p5}$_1$ and \eqref{p6}$_1$ by $h$ and $\theta$ respectively and integrating over $\Omega$ yields
	\begin{multline}\label{eq11}
		\frac{1}{2}\Dt \int_{\Omega}h^2\, \ud x - (\nu + \nu_r)\int_{\Omega}\triangle h \cdot h\, \ud x = \\
		-\int_{\Omega} \nabla q\cdot h\, \ud x  - \int_{\Omega} v \cdot \nabla h \cdot h\, \ud x - \int_{\Omega} h \cdot \nabla v \cdot h\, \ud x + 2\nu_r\int_{\Omega}\Rot \theta \cdot h\, \ud x + \int_{\Omega} f_{,x_3} \cdot h\, \ud x
	\end{multline}
	and
	\begin{multline}\label{eq13}
		\frac{1}{2}\Dt \int_{\Omega}\theta^2\, \ud x - \alpha \int_{\Omega} \triangle \theta \cdot \theta\, \ud x - \beta \int_{\Omega}\nabla \Div \theta \cdot \theta\, \ud x + 4\nu_r \int_{\Omega}\theta^2\, \ud x \\
		= - \int_{\Omega} h\cdot \nabla \omega \cdot \theta - \int_{\Omega} v \cdot \nabla \theta \cdot \theta\, \ud x + 2 \nu_r \int_{\Omega} \Rot h \cdot \theta\, \ud x + \int_{\Omega} g_{,x_3} \cdot \theta\, \ud x.
	\end{multline}
	Consider first the terms containing the Laplace operator. From Lemma \ref{l1} it follows that
	\begin{equation*}
		-\int_{\Omega} \triangle h \cdot h\, \ud x = \int_{\Omega} \Rot\Rot h \cdot h\, \ud x = \int_{\Omega} \abs{\Rot h}^2\, \ud x + \int_S \Rot h \times n \cdot h\, \ud S = \int_{\Omega} \abs{\Rot h}^2\, \ud x, \\
	\end{equation*}
	because $\Rot h \times n\vert_{S_1} = 0$ and on $S_2$ the equality $(\Rot h \times n)_3 = 0$ holds. Again, from Lemma \ref{l1} we have
	\begin{multline*}
		-\int_{\Omega} \triangle \theta\cdot \theta\, \ud x = \int_{\Omega} \Rot\Rot \theta \cdot \theta\, \ud x - \int_{\Omega} \nabla\Div \theta\cdot \theta\,\ud x \\
		= \int_{\Omega} \abs{\Rot \theta}^2\, \ud x + \int_S \Rot \theta\times n \cdot \theta\, \ud S + \int_{\Omega} \abs{\Div \theta}^2\, \ud x - \int_S (\theta \cdot n)\Div \theta\, \ud S
	\end{multline*}
	and
	\begin{equation*}
		\int_{\Omega} \Rot \theta \cdot h\, \ud x = \int_{\Omega} \Rot h \cdot \theta + \int_S \theta\times n \cdot h\, \ud S = \int_{\Omega} \Rot h\cdot \theta\, \ud x
	\end{equation*}
	because of the boundary conditions \eqref{p2}$_{3,4}$. 
	For the nonlinear terms the equalities 
	\begin{align*}
		\int_{\Omega} v\cdot \nabla h\cdot h\, \ud x &= -\frac{1}{2}\int_{\Omega} \Div v \abs{h}^2\, \ud x + \frac{1}{2}\int_S \abs{h}^2 v\cdot n\, \ud S = 0, \\
		\int_{\Omega} v\cdot \nabla \theta\cdot \theta\, \ud x &= -\frac{1}{2}\int_{\Omega} \Div v \abs{\theta}^2\, \ud x + \frac{1}{2}\int_S \abs{\theta}^2 v\cdot n\, \ud S = 0
	\end{align*}
	hold because of \eqref{p2}$_1$. Finally, integration by parts yields
	\begin{align*}
		-\int_{\Omega} \nabla \Div \theta \cdot \theta\, \ud x &= \int_{\Omega} \abs{\Div \theta}^2\, \ud x - \int_S \Div \theta (\theta \cdot n)\, \ud S = \int_{\Omega} \abs{\Div \theta}^2\, \ud x, \\
		\int_{\Omega} \nabla q \cdot h\, \ud x &= - \int_{\Omega} q \Div h\, \ud x + \int_S q (h \cdot n)\, \ud S = \int_{S_2} f_3 h_3\, \ud S = 0,
	\end{align*}
	where to justify the last equality we use Remark \ref{rem3}. Now, adding both sides in \eqref{eq11} and \eqref{eq13} and taking into account the above integration we obtain
	\begin{multline*}
		\frac{1}{2}\Dt \left(\int_{\Omega}h^2\, \ud x + \int_{\Omega}\theta^2\, \ud x\right) + \left(\nu + \nu_r\right)\int_{\Omega} \abs{\Rot h}^2\, \ud x + \alpha \int_{\Omega} \abs{\Rot \theta}^2\, \ud x + \left(\alpha + \beta\right) \int_{\Omega}\abs{\Div \theta}^2\, \ud x \\
+ 4\nu_r \int_{\Omega}\theta^2\, \ud x =  - \int_{\Omega} h\cdot \nabla v \cdot h\, \ud x - \int_{\Omega} h\cdot \nabla \omega \cdot \theta  + 4 \nu_r \int_{\Omega} \Rot h \cdot \theta\, \ud x \\
	+ \int_{\Omega} f_{,x_3} \cdot h\, \ud x + \int_{\Omega} g_{,x_3} \cdot \theta\, \ud x =: \sum_{k = 1}^5 I_k.
	\end{multline*}
	We estimate every term on the right-hand side by the means of the H\"older and the Young inequa\-lities:
	\begin{align*}
		I_1 &\leq \norm{h}_{L_6}\norm{\nabla v}_{L_2(\Omega)}\norm{h}_{L_3(\Omega)} \leq \epsilon_1 \norm{h}^2_{L_6(\Omega)} + \frac{1}{4\epsilon_1} \norm{\nabla v}_{L_2(\Omega)}^2\norm{h}^2_{L_3(\Omega)}, \\
		I_2 &\leq \norm{h}_{L_3(\Omega)}\norm{\nabla \omega}_{L_2(\Omega)}\norm{\theta}_{L_6(\Omega)} \leq \epsilon_2 \norm{\theta}_{L_6(\Omega)}^2 + \frac{1}{4\epsilon_2}\norm{\nabla \omega}_{L_2(\Omega)}^2\norm{h}_{L_3(\Omega)}^2, \\
		I_3 &\leq 4\nu_r \norm{\Rot h}_{L_2(\Omega)}\norm{\theta}_{L_2(\Omega)} \leq 4\nu_r\epsilon_3 \norm{\Rot h}_{L_2(\Omega)}^2 + \frac{\nu_r}{\epsilon_3}\norm{\theta}_{L_2(\Omega)}^2, \\
		I_4 &\leq \norm{f_{,x_3}}_{L_{\frac{6}{5}}(\Omega)}\norm{h}_{L_6(\Omega)} \leq \epsilon_5 \norm{h}_{L_6(\Omega)}^2 + \frac{1}{4\epsilon_5} \norm{f_{,x_3}}_{L_{\frac{6}{5}}(\Omega)}^2, \\
		I_5 &\leq \norm{g_{,x_3}}_{L_{\frac{6}{5}}(\Omega)}\norm{\theta}_{L_6(\Omega)} \leq \epsilon_6 \norm{\theta}_{L_6(\Omega)}^2 + \frac{1}{4\epsilon_6} \norm{g_{,x_3}}_{L_{\frac{6}{5}}(\Omega)}^2.
	\end{align*}
	Now we set $\epsilon_3 = \frac{1}{4}$. Then we estimate $\nu\norm{\Rot h}_{L_2(\Omega)}^2$ from below by $\frac{\nu}{c_{\Omega}}\norm{h}_{H^1(\Omega)}^2$ (see Remark \ref{rem7}) and $\alpha\big(\norm{\Rot \theta}_{L_2(\Omega)}^2 + \norm{\Div \theta}_{L_2(\Omega)}^2\big)$ from below by $\frac{\alpha}{c_{\Omega}}\norm{\theta}^2_{H^1(\Omega)}$ (see Lemma \ref{l2}, $\theta\cdot n\vert_S = 0$). Next we set 
	\begin{equation*}
		\epsilon_1 c_I = \epsilon_5 c_I = \frac{\nu}{4c_{\Omega}}, \qquad \epsilon_2 c_I = \epsilon_6 c_I = \frac{\alpha}{4c_{\Omega}}.
	\end{equation*}
	Summarizing,
	\begin{multline}\label{eq16}
		\frac{1}{2}\Dt \left(\int_{\Omega}h^2\, \ud x + \int_{\Omega}\theta^2\, \ud x\right) + \frac{\nu}{2c_{\Omega}}\norm{h}^2_{H^1(\Omega)} + \frac{\alpha}{2c_{\Omega}} \norm{\theta}^2_{H^1(\Omega)} \\
		\leq \frac{c_{I,\Omega}}{\nu}\norm{\nabla v}^2_{L_2(\Omega)}\norm{h}^2_{L_3(\Omega)} + \frac{c_{I,\Omega}}{\alpha}\norm{\nabla \omega}_{L_2(\Omega)}^2\norm{h}_{L_3(\Omega)}^2 	+ \frac{c_{I,\Omega}}{\nu}\norm{f_{,x_3}}^2_{L_{\frac{6}{5}}(\Omega)} + \frac{c_{I,\Omega}}{\alpha}\norm{g_{,x_3}}^2_{L_{\frac{6}{5}}(\Omega)}.
	\end{multline}
	Multiplying by $2$ and integrating with respect to $t \in (t_0,t_1)$ yields
	\begin{multline*}
		\norm{h(t)}^2_{L_2(\Omega)} + \norm{\theta(t)}^2_{L_2(\Omega)} + \frac{\nu}{c_{\Omega}}\int_{t_0}^t\norm{h(s)}^2_{H^1(\Omega)}\, \ud s + \frac{\alpha}{c_{\Omega}} \int_{t_0}^t\norm{\theta(s)}^2_{H^1(\Omega)}\, \ud s \\
		\leq \frac{c_{I,\Omega}}{\nu}\int_{t_0}^t\norm{\nabla v(s)}^2_{L_2(\Omega)}\norm{h(s)}^2_{L_3(\Omega)}\, \ud s + \frac{c_{I,\Omega}}{\alpha}\int_{t_0}^t\norm{\nabla \omega(s)}_{L_2(\Omega)}^2\norm{h(s)}_{L_3(\Omega)}^2\, \ud s \\
		+ \frac{c_{I,\Omega}}{\nu}\norm{f_{,x_3}}^2_{L_2(t_0,t;L_{\frac{6}{5}}(\Omega))} + \frac{c_{I,\Omega}}{\alpha}\norm{g_{,x_3}}^2_{L_2(t_0,t;L_{\frac{6}{5}}(\Omega))} + \norm{h(t_0)}^2_{L_2(\Omega)} + \norm{\theta(t_0)}^2_{L_2(\Omega)}.
	\end{multline*}
	Next we estimate the left-hand side from below by
	\begin{equation*}
		\min\left\{\frac{\nu}{c_{\Omega}}, \frac{\alpha}{c_{\Omega}},1\right\} \left(\norm{h}^2_{V_2^0(\Omega^t)} + \norm{\theta}^2_{V_2^0(\Omega^t)}\right).
	\end{equation*}
	In view of Lemma \ref{l4} we get that
	\begin{multline*}
		\min\left\{\frac{\nu}{c_{\Omega}}, \frac{\alpha}{c_{\Omega}},1\right\} \left(\norm{h}^2_{V_2^0(\Omega^t)} + \norm{\theta}^2_{V_2^0(\Omega^t)}\right) \\
		\leq \frac{2c_{I,\Omega}}{\min\{\alpha,\nu\}} \left(E^2_{v,\omega}(t) \norm{h}^2_{L_{\infty}(t_0,t;L_3(\Omega))} + \norm{f_{,x_3}}^2_{L_2(t_0,t;L_{\frac{6}{5}}(\Omega))} + \norm{g_{,x_3}}^2_{L_2(t_0,t;L_{\frac{6}{5}}(\Omega))}\right) \\
		+ \norm{h(t_0)}^2_{L_2(\Omega)} + \norm{\theta(t_0)}^2_{L_2(\Omega)},
	\end{multline*}
	which concludes the proof.
\end{proof}

\begin{rem}\label{rem4}
	In the above Lemma we used Lemma \ref{l2}, which implied that $\frac{\alpha}{c_{\Omega}}\norm{\theta}_{H^1(\Omega)}^2 \leq \alpha (\norm{\Rot \theta}_{L_2(\Omega)}^2 + \norm{\Div \theta}_{L_2(\Omega)}^2)$. In further considerations we need slightly different estimate. Since
	\begin{equation*}
		\alpha \norm{\Rot \theta}^2_{L_2(\Omega)} + (\alpha + \beta)\norm{\Div \theta}^2_{L_2(\Omega)} \geq \frac{\alpha}{2}\left(\norm{\Rot \theta}^2_{L_2(\Omega)} + \frac{\alpha}{2}\norm{\Rot \theta}^2_{L_2(\Omega)}\right) + \frac{\alpha}{2}\norm{\Rot \theta}^2_{L_2(\Omega)}
	\end{equation*}
	we infer from Lemma \ref{l2} that
	\begin{equation*}
		\alpha \norm{\Rot \theta}^2_{L_2(\Omega)} + (\alpha + \beta)\norm{\Div \theta}^2_{L_2(\Omega)} \geq \frac{\alpha}{2c_{\Omega}}\norm{\theta}_{H^1(\Omega)}^2 + \frac{\alpha}{2}\norm{\Rot \theta}^2_{L_2(\Omega)}.
	\end{equation*}
	By the Poincar\'e inequality (see Remark \ref{rem1}) and by the interpolation inequality we obtain
	\begin{equation*}
		\norm{h}_{L_3(\Omega)} \leq \norm{h}_{2}^{\frac{1}{2}}\norm{h}_{L_6(\Omega)}^{\frac{1}{2}} \leq c_{I,P} \norm{h}_{H^1(\Omega)}.
	\end{equation*}
	Thus, \eqref{eq16} can be written in a following form
	\begin{multline*}
		\frac{1}{2}\Dt \left(\int_{\Omega}h^2\, \ud x + \int_{\Omega}\theta^2\, \ud x\right) + \frac{\nu}{2c_{\Omega}}\norm{h}^2_{H^1(\Omega)} + \frac{\alpha}{2} \norm{\Rot \theta}^2_{L_2(\Omega)} \\
		\leq \frac{c_{I,P,\Omega}}{\nu}\norm{h}^2_{H^1(\Omega)}\left(\norm{\nabla v}^2_{L_2(\Omega)} + \norm{\nabla \omega}_{L_2(\Omega)}^2\right) + \frac{c_{I,\Omega}}{\nu}\norm{f_{,x_3}}^2_{L_{\frac{6}{5}}(\Omega)} + \frac{c_{I,\Omega}}{\alpha}\norm{g_{,x_3}}^2_{L_{\frac{6}{5}}(\Omega)}.
	\end{multline*}
	Utilizing Remark \ref{rem7} and multiplying the above inequality by $\frac{2\nu_r}{\alpha}$ leads to
	\begin{multline}\label{eq17}
		\frac{\nu_r}{\alpha}\Dt \left(\int_{\Omega}h^2\, \ud x + \int_{\Omega}\theta^2\, \ud x\right) + \nu_r\norm{\Rot \theta}^2_{L_2(\Omega)} \\
		\leq c_{\alpha,\nu,\nu_r,I,P,\Omega}\norm{\Rot h}^2_{L_2(\Omega)}\left(\norm{\nabla v}^2_{L_2(\Omega)} + \norm{\nabla \omega}_{L_2(\Omega)}^2\right) \\
		+ c_{\alpha,\nu,\nu_r,I,\Omega}\left(\norm{f_{,x_3}}^2_{L_{\frac{6}{5}}(\Omega)} + \norm{g_{,x_3}}^2_{L_{\frac{6}{5}}(\Omega)}\right).
	\end{multline}
	For now the above inequality seems to be useless, but later we shall see that it is necessary to obtain an estimate for $h$ in $V_2^1(\Omega)$ (see Lemma \ref{lem6}). 
\end{rem}

Finally we present the last Lemma of this Section. 

\begin{lem}\label{lem9}
	Let $E_{v,\omega}(t) < \infty$. Suppose that $\norm{h}_{L_\infty(t_0,t;L_3(\Omega))} < \infty$. Assume that $f' = (f_1,f_2) \in L_2(\Omega^t)$ and $v(t_0) \in H^1(\Omega)$. Then
	\begin{equation*}
		\norm{\chi}_{V_2^0(\Omega^t)} \leq c_{\alpha,\nu,\nu_r,I,P,\Omega} E_{v,\omega}(t)\norm{h}_{L_\infty(t_0,t;L_3(\Omega))} + c_{\nu,\nu_r}\norm{f'}_{L_2(\Omega^t)} + c_{\alpha,\nu,\nu_r,I,\Omega} E_{v,\omega}(t) + \norm{v(t_0)}_{H^1(\Omega)}.
	\end{equation*}
\end{lem}

\begin{proof}
	Multiplying \eqref{p9}$_1$ by $\chi$ and integrating over $\Omega$ yields
	\begin{multline*}
		\frac{1}{2}\Dt \int_{\Omega}\chi^2\, \ud x + (\nu + \nu_r)\int_{\Omega}\abs{\nabla \chi}^2\, \ud x = \int_{\Omega} h_3\chi\chi\, \ud x  -\int_{\Omega} h_2v_{3,x_1}\chi\, \ud x + \int_{\Omega} h_1v_{3,x_2}\chi\, \ud x  \\
		+ \int_{\Omega} F_3\chi\, \ud x + 2\nu_r\int_{\Omega}\left((\Rot \omega)_{2,x_1} - (\Rot\omega)_{1,x_2}\right)\chi\, \ud x.
	\end{multline*}
	The first three integrals on the right-hand side we estimate in the same way by the means of the H\"older and the Young inequalities:
	\begin{equation*}
		\int_{\Omega} h_3\chi\chi\, \ud x \leq \norm{h_3}_{L_3(\Omega)}\norm{\chi}_{L_2(\Omega)}\norm{\chi}_{L_6(\Omega)} \leq \epsilon_1\norm{\chi}^2_{L_6(\Omega)} + \frac{1}{4\epsilon_1}\norm{h_3}^2_{L_3(\Omega)}\norm{\chi}_{L_2(\Omega)}^2,
	\end{equation*}
	\begin{equation*}
		\int_{\Omega} h_2v_{3,x_1}\chi\, \ud x \leq \norm{h_2}_{L_3(\Omega)}\norm{v_{3,x_1}}_{L_2(\Omega)}\norm{\chi}_{L_6(\Omega)} \leq \epsilon_2\norm{\chi}^2_{L_6(\Omega)} + \frac{1}{4\epsilon_2}\norm{h_2}^2_{L_3(\Omega)}\norm{v_{3,x_1}}_{L_2(\Omega)}^2,
	\end{equation*}
	and
	\begin{equation*}
		\int_{\Omega} h_1v_{3,x_2}\chi\, \ud x \leq \norm{h_1}_{L_3(\Omega)}\norm{v_{3,x_2}}_{L_2(\Omega)}\norm{\chi}_{L_6(\Omega)} \leq \epsilon_3\norm{\chi}^2_{L_6(\Omega)} + \frac{1}{4\epsilon_3}\norm{h_1}^2_{L_3(\Omega)}\norm{v_{3,x_2}}_{L_2(\Omega)}^2.
	\end{equation*}
	In the last two integrals we first integrate by parts and then use the H\"older and the Young inequalities. Note that the boundary integrals are equal to zero due to the boundary conditions for $\chi$. 
	\begin{multline*}
		\int_{\Omega}F_3\chi\, \ud x = - \int_{\Omega} f_2\chi_{,x_1} - f_1\chi_{,x_2}\, \ud x + \int_S f_2\chi n_1 - f_1\chi n_2\,\ud S \\
		\leq \norm{f'}_{L_2(\Omega)}\norm{\nabla '\chi}_{L_2(\Omega)} \leq \epsilon_4 \norm{\nabla \chi}^2_{L_2(\Omega)} + \frac{1}{4\epsilon_4}\norm{f'}_{L_2(\Omega)}^2
	\end{multline*}
	and
	\begin{multline*}
		\int_{\Omega}\left((\Rot \omega)_{2,x_1} - (\Rot\omega)_{1,x_2}\right)\chi\, \ud x = - \int_{\Omega} (\Rot \omega)_2\chi_{,x_1} - (\Rot \omega)_1\chi_{,x_2}\, \ud x \\
		+ \int_S(\Rot \omega)_2\chi n_1 - (\Rot\omega)_1\chi n_2\, \ud S\\
		\leq \norm{(\Rot \omega)'}_{L_2(\Omega)}\norm{\nabla' \chi}_{L_2(\Omega)} \leq \epsilon_5 \norm{\nabla \chi}_{L_2(\Omega)}^2 + \frac{1}{4\epsilon_5}\norm{(\Rot \omega)'}_{L_2(\Omega)}^2.
	\end{multline*}
	Since $\chi\vert_{S_1} = 0$, we can use the Poincar\'e inequality $\norm{\chi}_{L_6(\Omega)} \leq c_I\norm{\chi}_{H^1(\Omega)} \leq c_{I,P} \norm{\nabla \chi}_{L_2(\Omega)}$. Hence, we put $\epsilon_1 = \epsilon_2 = \epsilon_3 = \frac{\nu + \nu_r}{12c_{I,P}^2}$ and $\epsilon_4 = 2\nu_r \epsilon_5 = \frac{\nu + \nu_r}{8}$. Then we see that
	\begin{multline*}
		\frac{1}{2}\Dt \int_{\Omega}\chi^2\, \ud x + \frac{\nu + \nu_r}{2}\int_{\Omega}\abs{\nabla \chi}^2\, \ud x \\
		\leq \frac{3c_{I,P}^2}{\nu + \nu_r}\left(\norm{h_3}^2_{L_3(\Omega)}\norm{\chi}_{L_2(\Omega)}^2 + \norm{h_2}^2_{L_3(\Omega)}\norm{v_{3,x_1}}_{L_2(\Omega)}^2 + \norm{h_1}^2_{L_3(\Omega)}\norm{v_{3,x_2}}_{L_2(\Omega)}^2\right) \\
		+ \frac{2}{\nu + \nu_r}\norm{f'}_{L_2(\Omega)}^2 + \frac{4\nu_r}{\nu + \nu_r}\norm{(\Rot \omega)'}_{L_2(\Omega)}^2.
	\end{multline*}
	Next we multiply by $2$, integrate with respect to $t \in (t_0,t_1)$ and use the energy estimate (Lemma \ref{l4}). It gives
	\begin{multline*}
		\min\left\{1,\nu + \nu_r\right\} \norm{\chi}_{V^0_2(\Omega^t)}^2 \\
		\leq \frac{6c_{I,P}^2\norm{h}_{L_\infty(t_0,t;L_3(\Omega))}^2}{\nu + \nu_r}\left(\norm{\chi}_{L_2(\Omega^t)}^2 + \norm{v_{3,x_1}}_{L_2(\Omega^t)}^2 + \norm{v_{3,x_2}}_{L_2(\Omega^t)}^2\right) \\
		+ \frac{4}{\nu + \nu_r}\norm{f'}_{L_2(\Omega^t)}^2 + \frac{8\nu_r}{\nu + \nu_r}\norm{(\Rot \omega)'}_{L_2(\Omega^2)}^2 + \norm{\chi(t_0)}^2_{L_2(\Omega)}.
	\end{multline*}
	To complete the proof we observe that $\norm{\chi(t_0)}_{L_2(\Omega)} \leq \norm{v(t_0)}_{H^1(\Omega)}$.
\end{proof}

\section{Higher order estimates}\label{sec3.2}

In this Section we confine our attention to find estimates for $v$ and $\omega$ in the norm of the space $W^{2,1}_2(\Omega^t)$ in terms of data only. Of course we cannot expect such a result without an additional smallness assumption. Therefore we shall make use of $\delta(t)$ which was introduced \eqref{eq26}. We emphasize that the smallness condition does not involve $L_2$-norms of the initial velocity and the initial microrotation fields. Below we briefly outline the reason behind that. 

The central aim is to estimate the nonlinear terms $v\cdot \nabla v$ and $v\cdot \nabla \omega$ in $L_2(\Omega^t)$. To accomplish it we first consider the problem for $h$ and seek for estimate of its solution in $V^1_2(\Omega^t)$. The estimate has a form 
\begin{equation}\label{eq53}
	\norm{h}_{V_2^1(\Omega^t)} \leq c\left(\exp\left(\norm{\nabla v}^2_{L_2(t_0,t;L_6(\Omega))}\right) + 1\right) \cdot \delta(t) 
\end{equation}
(see Lemma \ref{lem6}). Its direct consequence is 
\begin{equation*}
	\norm{h}_{L_\infty(t_0,t;H^1(\Omega^t))} \leq c\left(\exp\left(\norm{\nabla v}^2_{L_2(t_0,t;L_6(\Omega))}\right) + 1\right) \cdot \delta(t),
\end{equation*}
which we use to bound $\norm{h}_{L_\infty(t_0,t;L_3(\Omega))}$ by the means of the interpolation between $L_p$ spaces and the Poincar\'e inequality. The reason becomes apparent if we take into account that the inequalities for $\norm{h}_{V^0_2(\Omega^t)}$ and $\norm{\chi}_{V^0_2(\Omega^t)}$ (see Lemmas \ref{lem8} and \ref{lem9}) are dependent on that norm. Therefore we can write
\begin{equation*}
	\norm{h}_{V^0_2(\Omega^t)} + \norm{\chi}_{V^0_2(\Omega^t)} \leq c\left(\exp\left(\norm{\nabla v}^2_{L_2(t_0,t;L_6(\Omega))}\right) + 1\right) \cdot \delta(t) + \text{data}.
\end{equation*}
Applying the above inequality and \eqref{eq53} and using the structure of the domain, which is a Cartesian product of two sets, we are able to estimate $\nabla v$ in the norm $V^0_2(\Omega^t)$ (see Lemma \ref{lem13}). Since the estimate has the form 
\begin{equation*}
	\norm{\nabla v}_{V^0_2(\Omega^t)} \leq c\left(\exp\left(\norm{\nabla v}^2_{L_2(t_0,t;L_6(\Omega))}\right) + 1\right) \cdot \delta(t) + \text{data}
\end{equation*}
we finally write
\begin{equation*}
	\norm{\nabla v}_{V^0_2(\Omega^t)} \leq \text{data}
\end{equation*}
if only $\delta(t)$ is small enough. The above inequality in connection with embedding theorem guarantees that the nonlinear terms can be estimated in the $L_2$-norm with respect to $\Omega^t$, thereby providing estimates for $v$ and $\omega$ in $W^{2,1}_2(\Omega^t)$, which is demonstrated in Lemmas \ref{lem25} and \ref{lem18}.

Let us now move on to th details of the proof:
\begin{lem}\label{lem6}
    Suppose that $f_3\vert_{S_2} = 0$, $g'\vert_{S_2} = 0$ and $\nabla v \in L_2(t_0,t;L_6(\Omega))$. Let $E_{v,\omega}(t) < \infty$ and $\delta(t) < \infty$. Then
	\begin{equation*}
		\norm{h}_{V^1_2(\Omega^t)} \leq c_{\alpha,\nu,\nu_r,I,P,\Omega} \left(\exp\left(\norm{\nabla v}^2_{L_2(t_0,t;L_6(\Omega))} + E_{v,\omega}^2(t)\right) + 1\right) \cdot \delta(t).
	\end{equation*}
\end{lem}

In the proof of the above Lemma, we use 
\begin{equation}\label{eq43}
	\norm{h}_{H^2(\Omega)} \leq c_{\Omega} \norm{\triangle h}_{L_2(\Omega)},
\end{equation}
which we prove in analogous manner as in Remark \ref{rem7}.

\begin{proof}[Proof of Lemma \ref{lem6}]
	We multiply \eqref{p5}$_1$ by $- \triangle h$ and integrate over $\Omega$, which yields
	\begin{multline}\label{eq9}
		- \int_{\Omega} h_{,t}\cdot \triangle h\, \ud x + (\nu + \nu_r)\int_{\Omega} \abs{\triangle h}^2\, \ud x - \int_{\Omega} \nabla q \cdot \triangle h\, \ud x \\
		= \int_{\Omega} v \cdot \nabla h \cdot \triangle h\, \ud x + \int_{\Omega} h \cdot \nabla v \cdot \triangle h\, \ud x - 2\nu_r \int_{\Omega} \Rot \theta \cdot \triangle h\, \ud x - \int_{\Omega} f,_{x_3}\cdot \triangle h\, \ud x.
	\end{multline}
	For the first term on the left-hand side we have
	\begin{multline*}
		- \int_{\Omega} h_{,t}\cdot \triangle h\, \ud x = \int_{\Omega} h_{,t} \cdot \Rot \Rot h\, \ud x \\
		= \frac{1}{2}\Dt \int_{\Omega}\abs{\Rot h}^2\, \ud x + \int_{S_1} h_{,t}\cdot \Rot h \times n\, \ud S_1 + \int_{S_2} h_{2,t} \left(\Rot h\right)_1 - h_{1,t} \left(\Rot h\right)_2\, \ud S_2 \\
		= \frac{1}{2}\Dt \int_{\Omega}\abs{\Rot h}^2\, \ud x
	\end{multline*}
	where the boundary integrals vanish due to the boundary conditions \eqref{p5}$_{3,4}$.

	The third term on the left-hand side in \eqref{eq9} is equal to
	\begin{multline*}
		\int_{\Omega} \nabla q\cdot \Rot\Rot h\, \ud x = \int_{\Omega} \Rot \nabla q \cdot \Rot h\, \ud x + \int_{S_1} \nabla q\cdot \Rot h \times n\, \ud S_1 \\
		+ \int_{S_2} q_{,x_2}\left(\Rot h\right)_1 - q_{,x_1}\left(\Rot h\right)_2\, \ud S_2 = 0,
	\end{multline*}
	which follows from the boundary condition \eqref{p5}$_3$, Remark \ref{rem3} and the assumption that $f_3\vert_{S_2} = 0$.

	Consider next the first term on the right-hand side in \eqref{eq9}. Since $\Div v = 0$ we may integrate by parts, which yields
	\begin{multline*}
		\int_{\Omega} v \cdot \nabla h \cdot \triangle h\, \ud x \leq \norm{\nabla v}_{L_6(\Omega)}\norm{\nabla h}_{L_3(\Omega)}\norm{\Rot h}_{L_2(\Omega)} + \int_S \left(v\cdot \nabla\right) h\cdot (\Rot h \times n)\, \ud S \\
		\leq \epsilon_1 c_I \norm{\nabla h}^2_{H^1(\Omega)} + \frac{1}{4\epsilon_1}\norm{\nabla v}^2_{L_6(\Omega)}\norm{\Rot h}^2_{L_2(\Omega)} \leq \epsilon_1 c_I \norm{h}^2_{H^2(\Omega)} + \frac{1}{4\epsilon_1}\norm{\nabla v}^2_{L_6(\Omega)}\norm{\Rot h}^2_{L_2(\Omega)} \\
		\leq \epsilon_1 c_{\Omega,I} \norm{\triangle h}^2_{L_2(\Omega)} + \frac{1}{4\epsilon_1}\norm{\nabla v}^2_{L_6(\Omega)}\norm{\Rot h}^2_{L_2(\Omega)},
	\end{multline*}
	where we used that $\norm{\nabla h}_{L_3(\Omega)} \leq \norm{\nabla h}^{\frac{1}{2}}_{L_2(\Omega)}\norm{\nabla h}^{\frac{1}{2}}_{L_6(\Omega)} \leq c_I \norm{\nabla h}_{H^1(\Omega)}$. The last inequality above is justified in light of \eqref{eq43}. 

	For the second term on the right-hand side in \eqref{eq9} we simply have
	\begin{equation*}
		\norm{h}_{L_3(\Omega)}\norm{\nabla v}_{L_6(\Omega)}\norm{\triangle h}_{L_2(\Omega)} \leq \epsilon_2 \norm{\triangle h}_{L_2(\Omega)}^2 + \frac{1}{4\epsilon_2} \norm{h}_{L_2(\Omega)}\norm{h}_{L_6(\Omega)}\norm{\nabla v}_{L_6(\Omega)}^2.
	\end{equation*}

	The third term is estimated as follows
	\begin{equation*}
		2\nu_r\int_{\Omega}\Rot \theta\cdot \triangle h\, \ud x \leq 2\nu_r\norm{\Rot \theta}_{L_2(\Omega)}\norm{\triangle h}_{L_2(\Omega)} \\
		\leq 2\nu_r\epsilon_3 \norm{\triangle h}^2_{L_2(\Omega)} + \frac{\nu_r}{2\epsilon_3} \norm{\Rot \theta}^2_{L_2(\Omega)}.
	\end{equation*}

	Finally, for the fourth term we have
	\begin{equation*}
		\int_{\Omega}f,_{x_3}\cdot \triangle h\, \ud x \leq \norm{f,_{x_3}}_{L_2(\Omega)}\norm{\triangle h}_{L_2(\Omega)} \leq \epsilon_4 \norm{\triangle h}^2_{L_2(\Omega)} + \frac{1}{4\epsilon_4} \norm{f,_{x_3}}^2_{L_2(\Omega)}.
	\end{equation*}
	
	Setting $\epsilon_1 c_{\Omega,I} = \epsilon_2 = \epsilon_4 = \frac{\nu}{6}$ and $\epsilon_3 = \frac{1}{2}$ yields
	\begin{multline}\label{eq31}
		\frac{1}{2}\Dt \int_{\Omega}\abs{\Rot h}^2\, \ud x + \frac{\nu}{2}\norm{\triangle h}^2_{L_2(\Omega)} \leq \frac{3c_{\Omega,I}}{2\nu}\norm{\nabla v}^2_{L_6(\Omega)}\norm{\Rot h}^2_{L_2(\Omega)} \\
		+ \frac{3}{2\nu} \norm{h}_{L_2(\Omega)}\norm{h}_{L_6(\Omega)}\norm{\nabla v}_{L_6(\Omega)}^2 + \nu_r \norm{\Rot \theta}^2_{L_2(\Omega)} + \frac{3}{2\nu}\norm{f_{,x_3}}^2_{L_2(\Omega)} \\
		\leq \frac{3c_{I,\Omega}}{\nu}\norm{\nabla v}^2_{L_6(\Omega)}\norm{\Rot h}^2_{L_2(\Omega)} + \nu_r \norm{\Rot \theta}^2_{L_2(\Omega)} + \frac{3}{2\nu}\norm{f_{,x_3}}^2_{L_2(\Omega)},
	\end{multline}
	where in the last inequality we used $\norm{h}_{L_2(\Omega)}\norm{h}_{L_6(\Omega)} \leq c_I\norm{h}_{H^1(\Omega)}^2$ and subsequently utilized Lemma \ref{l2}. Next we add inequality \eqref{eq17} (see Remark \ref{rem4}), which leads to
	\begin{multline*}
		\frac{1}{2}\Dt \int_{\Omega}\abs{\Rot h}^2\, \ud x + \frac{\nu_r}{\alpha}\Dt \left(\int_{\Omega}h^2\, \ud x + \int_{\Omega}\theta^2\, \ud x\right) + \nu_r\norm{\Rot \theta}^2_{L_2(\Omega)}+ \frac{\nu}{2}\norm{\triangle h}^2_{L_2(\Omega)}  \\
		\leq \frac{3c_{I,\Omega}}{\nu}\norm{\nabla v}^2_{L_6(\Omega)}\norm{\Rot h}^2_{L_2(\Omega)} + \nu_r \norm{\Rot \theta}^2_{L_2(\Omega)} + \frac{3}{2\nu}\norm{f_{,x_3}}^2_{L_2(\Omega)} \\
		+ c_{\alpha,\nu,\nu_r,I,P,\Omega}\norm{\Rot h}^2_{L_2(\Omega)}\left(\norm{\nabla v}^2_{L_2(\Omega)} + \norm{\nabla \omega}_{L_2(\Omega)}^2\right) \\
		+ c_{\alpha,\nu,\nu_r,I,\Omega}\left(\norm{f_{,x_3}}^2_{L_{\frac{6}{5}}(\Omega)} + \norm{g_{,x_3}}^2_{L_{\frac{6}{5}}(\Omega)}\right).
	\end{multline*}
	Multiplying by $2$ and integrating with respect to $t \in (t_0,t)$ yields
	\begin{multline}\label{eq18}
		\norm{\Rot h(t)}^2_{L_2(\Omega)} + \frac{2\nu_r}{\alpha}\norm{h(t)}^2_{L_2(\Omega)} + \frac{2\nu_r}{\alpha}\norm{\theta(t)}^2_{L_2(\Omega)} + \nu\norm{\triangle h}^2_{L_2(\Omega^t)}  \\
		\leq \frac{6c_{I,\Omega}}{\nu}\int_{t_0}^t\norm{\nabla v(s)}^2_{L_6(\Omega)}\norm{\Rot h(s)}^2_{L_2(\Omega)}\, \ud s + \frac{3}{\nu}\norm{f_{,x_3}}^2_{L_2(\Omega^t)} \\
		+ c_{\alpha,\nu,\nu_r,I,P,\Omega}\int_{t_0}^t\left(\norm{\Rot h(s)}^2_{L_2(\Omega)}\left(\norm{\nabla v(s)}^2_{L_2(\Omega)} + \norm{\nabla \omega(s)}_{L_2(\Omega)}^2\right)\right)\, \ud s \\
		+ c_{\alpha,\nu,\nu_r,I,\Omega}\left(\norm{f_{,x_3}}^2_{L_2(t_0,t;L_{\frac{6}{5}}(\Omega))} + \norm{g_{,x_3}}^2_{L_2(t_0,t;L_{\frac{6}{5}}(\Omega))}\right) \\
		+ \norm{\Rot h(t_0)}^2_{L_2(\Omega)} + \frac{2\nu_r}{\alpha}\norm{h(t_0)}^2_{L_2(\Omega)} + \frac{2\nu_r}{\alpha}\norm{\theta(t_0)}^2_{L_2(\Omega)}.
	\end{multline}
	From the Gronwall inequality and the energy estimates (Lemma \ref{l4}) it follows that
	\begin{multline*}
		\norm{\Rot h}^2_{L_\infty(t_0,t;L_2(\Omega))} \leq c_{\alpha,\nu,\nu_r,I,P,\Omega}\exp\left(\norm{\nabla v}^2_{L_2(t_0,t;L_6(\Omega))} + E_{v,\omega}^2(t)\right)\\
		\cdot \bigg(\norm{f_{,x_3}}^2_{L_2(\Omega^t)} + \norm{f_{,x_3}}^2_{L_2(t_0,t;L_{\frac{6}{5}}(\Omega))} + \norm{g_{,x_3}}^2_{L_2(t_0,t;L_{\frac{6}{5}}(\Omega))} \\
		+ \norm{\Rot h(t_0)}^2_{L_2(\Omega)} + \norm{h(t_0)}^2_{L_2(\Omega)} + \norm{\theta(t_0)}^2_{L_2(\Omega)}\bigg).
	\end{multline*}
	By the H\"older inequality we write
	\begin{multline*}
		\norm{\Rot h}^2_{L_\infty(t_0,t;L_2(\Omega))} \leq c_{\alpha,\nu,\nu_r,I,P,\Omega}\exp\left(\norm{\nabla v}^2_{L_2(t_0,t;L_6(\Omega))} + E_{v,\omega}^2(t)\right) \\
		\cdot \bigg(\norm{f_{,x_3}}^2_{L_2(\Omega^t)} + \norm{g_{,x_3}}^2_{L_2(\Omega^t)} + \norm{\Rot h(t_0)}^2_{L_2(\Omega)} + \norm{h(t_0)}^2_{L_2(\Omega)} + \norm{\theta(t_0)}^2_{L_2(\Omega)}\bigg).
	\end{multline*}
	Putting the above estimate in \eqref{eq18} and using the H\"older inequality yields
	\begin{multline*}
		\norm{\Rot h(t)}_{L_2(\Omega)}^2 + \nu\norm{\triangle h}^2_{L_2(\Omega^t)} \\
		\leq c_{\alpha,\nu,\nu_r,I,P,\Omega}\norm{\Rot h}^2_{L_\infty(t_0,t;L_2(\Omega))}\left(\norm{\nabla v}^2_{L_2(t_0,t;L_6(\Omega))} + \norm{\nabla v}^2_{L_2(\Omega^t)} + \norm{\nabla \omega}^2_{L_2(\Omega^t)}\right) \\
		+ c_{\alpha,\nu,\nu_r,I,\Omega}\left(\norm{f_{,x_3}}^2_{L_2(\Omega^t)} + \norm{g_{,x_3}}^2_{L_2(\Omega^t)} + \norm{\Rot h(t_0)}^2_{L_2(\Omega)} + \norm{h(t_0)}^2_{L_2(\Omega)} + \norm{\theta(t_0)}^2_{L_2(\Omega)}\right).
	\end{multline*}
	In light of Remark \ref{rem7} and \eqref{eq43} we estimate the left-hand side  from below by $\min\{1,\nu\}\norm{h}^2_{V^1_2(\Omega^t)}$. Thus
	\begin{multline*}
		\norm{h}^2_{V^1_2(\Omega^t)} \leq c_{\alpha,\nu,\nu_r,I,P,\Omega} \exp\left(\norm{\nabla v}^2_{L_2(t_0,t;L_6(\Omega))} + E_{v,\omega}^2(t)\right) \cdot \delta(t) \\
		\cdot \left(\norm{\nabla v}^2_{L_2(t_0,t;L_6(\Omega))} + E^2_{v,\omega}(t)\right) + c_{\alpha,\nu,\nu_r,I,\Omega}\delta(t)
	\end{multline*}
	An obvious inequality $e^xx \leq e^{2x}$ for $x \geq 0$ ends the proof.
\end{proof}

\begin{lem}\label{lem13}
	Let $\chi, h_3 \in V_2^0(\Omega^t)$, $f' = (f_1,f_2) \in L_2(\Omega^t)$ and $v(t_0) \in H^1(\Omega)$. Let $E_{v,\omega}(t) < \infty$, $E_{h,\theta}(t) < \infty$. Suppose that $\delta(t)$ is small enough. Then $\nabla v \in V_2^0(\Omega^t)$ and the estimate
	\begin{equation*}
		\norm{\nabla v}_{V^0_2(\Omega^t)} \leq c_{\alpha,\nu,\nu_r,I,P,\Omega} \left(E_{v,\omega}(t) + E_{h,\theta}(t) + \norm{f'}_{L_2(\Omega^t)} + \norm{v(t_0)}_{H^1(\Omega)}\right).
	\end{equation*}
	holds.
\end{lem}

\begin{proof}
	Consider the problem
	\begin{equation*}
		\begin{aligned}
			&v_{2,x_1} - v_{1,x_2} = \chi & &\text{in $\Omega'$}, \\
			&v_{1,x_1} + v_{2,x_2} = -h_3 & &\text{in $\Omega'$},  \\
			&v' \cdot n' = 0 & &\text{on $S_1'$}, 
		\end{aligned}
	\end{equation*}
	where by $\Omega'$ and $S'$ we understand the sets $\Omega \cap \{-a < x_3 < a\colon x_3 = \text{const}\}$ and $S' = S_1 \cap \{-a < x_3 < a\colon x_3 = \text{const}\}$.
	For solutions to this problem we have estimates
	\begin{align*}
		\norm{\nabla' v(x_3)}_{L_2(\Omega')}^2 &\leq c_{\Omega'}\left(\norm{\chi(x_3)}^2_{L_2(\Omega')} + \norm{h_3(x_3)}^2_{L_2(\Omega')}\right), \\
		\norm{\nabla' v(x_3)}_{H^1(\Omega')}^2 &\leq c_{\Omega'}\left(\norm{\chi(x_3)}^2_{H^1(\Omega')} + \norm{h_3(x_3)}^2_{H^1(\Omega')}\right).
	\end{align*}
	Integrating both inequalities with respect to $x_3 \in (-a,a)$ ant taking the $L_2$- and $L_\infty$-norms with respect to $t \in (t_0,t_1)$ yields
	\begin{equation*}
		\norm{\nabla' v}_{V_2^0(\Omega^t)}^2 \leq c_{\Omega'}\left(\norm{\chi}^2_{V_2^0(\Omega^t)} + \norm{h_3}_{V_2^0(\Omega^t)}^2\right).
	\end{equation*}
	Using the estimate (see Lemma \ref{lem8})
	\begin{equation*}
		\norm{h}_{V_2^0(\Omega^t)}^2 \leq c_{\alpha,\nu,I,\Omega} \left(E^2_{v,\omega}(t) \norm{h}_{L_{\infty}(t_0,t;L_3(\Omega))}^2 + E^2_{h,\theta}(t)\right)
	\end{equation*} 
	we obtain
	\begin{multline*}
		\norm{\nabla v}_{V^0_2(\Omega^t)}^2 \leq c_{\alpha,\nu,I,\Omega} \left(E_{v,\omega}^2(t) \norm{h}_{L_{\infty}(t_0,t;L_3(\Omega))}^2 + E^2_{h,\theta}(t)\right) \\
		+ c_{\nu,\nu_r,\alpha,\Omega,I,P} E_{v,\omega}^2(t)\norm{h}_{L_\infty(t_0,t;L_3(\Omega))}^2 + c_{\nu,\nu_r}\norm{f'}_{L_2(\Omega^t)}^2 + c_{\alpha,\nu,\nu_r,I,\Omega} E_{v,\omega}^2(t) + \norm{v(t_0)}_{H^1(\Omega)}^2.
	\end{multline*}
	The interpolation and the Poincar\'e inequalities (see Remark \ref{rem1}) imply that
	\begin{equation*}
		\norm{h}_{L_\infty(t_0,t;L_3(\Omega))} \leq \norm{h}_{L_\infty(t_0,t;L_2(\Omega))}^{\frac{1}{2}}\norm{h}_{L_\infty(t_0,t;L_6(\Omega))}^{\frac{1}{2}} \leq c_{I,P} \norm{h}_{L_\infty(t_0,t;H^1(\Omega))}.
	\end{equation*}
	From Lemma \ref{lem11} we infer that
	\begin{equation*}
		\norm{\nabla v}_{L_2(t_0,t;L_6(\Omega))} \leq c_{I} \norm{\nabla v}_{V^0_2(\Omega^t)}.
	\end{equation*}
	Finally
	\begin{multline*}
		\norm{\nabla v}^2_{V^0_2(\Omega^t)} \leq c_{\alpha,\nu,\nu_r,I,P,\Omega} E^2_{v,\omega}(t) \left(\exp\left(\norm{\nabla v}^2_{V^0_2(\Omega^t)} + E^2_{v,\omega}(t)\right) + 1\right) \cdot \delta(t) \\
		+ E^2_{v,\omega}(t) + E^2_{h,\theta}(t) + \norm{f'}_{L_2(\Omega^t)}^2 + \norm{v(t_0)}_{H^1(\Omega)}^2.
	\end{multline*}
	Taking $\delta(t)$ sufficiently small ends the proof.
\end{proof}

\begin{rem}\label{rem6}
	From the above Lemma two natural embeddings follow (see Lem. 3.7 in \cite{wm5}). In view of Lemma \ref{lem11} we see that $\nabla v \in L_q(t_0,t;L_p(\Omega))$ and
	\begin{equation*}
		\norm{\nabla v}_{L_q(t_0,t;L_p(\Omega))} \leq c_I \norm{\nabla v}_{V_2^0(\Omega^t)}
	\end{equation*}
	holds for $p,q$ satisfying $\frac{3}{p} + \frac{2}{q} = \frac{3}{2}$. Setting $p = q$ we get $\frac{5}{p} = \frac{3}{2} \Leftrightarrow p = \frac{10}{3}$. Hence
	\begin{equation*}
		\norm{\nabla v}_{L_{\frac{10}{3}}(\Omega^t)} \leq c_I \norm{\nabla v}_{V_2^0(\Omega^t)}.
	\end{equation*}
	From the Sobolev embedding lemma we know that $W^1_p(\Omega) \hookrightarrow L_{\frac{3p}{3 - p}}(\Omega)$ for $p < 3$. Thus,
	\begin{equation*}
		\norm{v}_{L_q(t_0,t;L_{\frac{3p}{3 - p}}(\Omega))} \leq c_I \norm{\nabla v}_{L_q(t_0,t;L_p(\Omega))} \leq c_I \norm{\nabla v}_{V^0_2(\Omega^t)}
	\end{equation*}
	for $\frac{3}{p} + \frac{2}{q} = \frac{3}{2}$. Let $r = \frac{3p}{3 - p}$ and $q = r$. Then $p = \frac{3r}{3 + r}$ and 
	\begin{equation*}
		\frac{3}{p} + \frac{2}{q} = \frac{3}{2} \Leftrightarrow \frac{3 + r}{r} + \frac{2}{r} = \frac{3}{2} \Leftrightarrow r = 10.
	\end{equation*}
	Hence
	\begin{equation*}
		\norm{v}_{L_{10}(\Omega^t)} \leq c_I\norm{\nabla v}_{V_2^0(\Omega^t)}. 
	\end{equation*}
\end{rem}

\begin{lem}\label{lem25}
	Let $E_{v,\omega}(t) < \infty$ and $E_{h,\theta}(t) < \infty$. Assume that $f,g\in L_2(\Omega^t)$, $v(t_0),\omega(t_0) \in H^1(\Omega)$. Then, for $\delta(t)$ sufficiently small $\omega \in W^{2,1}_{2}(\Omega^t)$ and the inequality
	\begin{equation*}
		\norm{\omega}_{W^{2,1}_{2}(\Omega^t)} \leq c_{\alpha,\nu,\nu_r,I,P,\Omega}\left(E_{v,\omega}(t) + E_{h,\theta}(t) + \norm{f'}_{L_2(\Omega^t)} + \norm{v(t_0)}_{H^1(\Omega)} + \norm{\omega(t_0)}_{H^1(\Omega)} + 1\right)^3 \\
	\end{equation*}
	holds.
\end{lem}

\begin{proof}
	Let us rewrite \eqref{p1}$_2$ and $\eqref{p2}$ in the form
	\begin{equation*}
		\begin{aligned}
			&\begin{aligned}
				&\omega_{,t} - \alpha \triangle \omega - \beta\nabla\Div\omega \\
				&\mspace{60mu} = - v\cdot \nabla \omega - 4\nu_r\omega + 2\nu_r\Rot v + g
			\end{aligned}& &\text{in } \Omega^t, \\
			&\omega = 0& &\text{on } S_1^t, \\
			&\omega' = 0, \qquad \omega_{3,x_3} = 0 & &\text{on } S_2^t, \\
			&\omega\vert_{t = t_0} = \omega(t_0) & &\text{in } \Omega\times\{t = t_0\}.
		\end{aligned}
	\end{equation*}
	Then, from Lemma \ref{lem16} it follows that
	\begin{multline}\label{eq19}
		\norm{\omega}_{W^{2,1}_2(\Omega^t)} \leq c_{\Omega}\Big(\norm{v\cdot \nabla \omega}_{L_2(\Omega^t)} + 4\nu_r \norm{\omega}_{L_2(\Omega^t)} + 2\nu_r \norm{\Rot v}_{L_2(\Omega^t)} \\
		+ \norm{g}_{L_2(\Omega^t)} + \norm{\omega(t_0)}_{H^1(\Omega)}\Big).
	\end{multline}
	From the H\"older inequality we get
	\begin{equation*}
		\norm{v\cdot \nabla \omega}_{L_2(\Omega^t)} \leq c_I\norm{v}_{L_{10}(\Omega^t)} \norm{\nabla \omega}_{L_{\frac{5}{2}}(\Omega^t)}.
	\end{equation*}
	To estimate $\norm{\nabla \omega}_{L_\frac{5}{2}(\Omega^t)}$ we use the interpolation theorem (see Lemma \ref{lem10}) with $p = 2$, $q = \frac{5}{2}$. It gives
	\begin{equation*}
		\norm{\nabla \omega}_{L_{\frac{5}{2}}(\Omega^t)} \leq c_1 \epsilon^{\frac{1}{2}} \norm{\omega}_{W^{2,1}_2(\Omega^t)} + c_2 \epsilon^{-\frac{1}{2}} \norm{\omega}_{L_2(\Omega^t)}.
	\end{equation*}
	Setting $c_1\epsilon^{\frac{1}{2}} = \frac{1}{2c_I\norm{v}_{L_{10}(\Omega^t)}}$ yields
	\begin{equation*}
		\norm{v\cdot \nabla \omega}_{L_2(\Omega^t)} \leq \frac{1}{2}\norm{\omega}_{W^{2,1}_2(\Omega^t)} + 2c_1c_2c_I \norm{v}^2_{L_{10}(\Omega^t)} \norm{\omega}_{L_2(\Omega^t)}.
	\end{equation*}

	To estimate the second and the third term on the right-hand side in \eqref{eq19} we observe that $\norm{\omega}_{L_2(\Omega^t)} \leq \norm{\omega}_{L_2(0,t;H^1(\Omega))}$ and use Lemma \ref{l4}. Finally, the inequality
	\begin{multline*}
		\norm{\omega}_{W^{2,1}_2(\Omega^t)} \leq c_{\alpha,\nu,\nu_r,I,P,\Omega}\bigg(E_{v,\omega}^3(t) + E_{h,\theta}^3(t) + \norm{f'}_{L_2(\Omega^t)} ^3 + \norm{v(t_0)}_{H^1(\Omega)}^3 \\
		+ E_{v,\omega}(t) + \norm{g}_{L_2(\Omega^t)} + \norm{\omega(t_0)}_{H^1(\Omega)}\bigg),
	\end{multline*}
	holds, which ends the proof.
\end{proof}

\begin{lem}\label{lem18}
	Let $E_{v,\omega}(t) < \infty$, $E_{h,\theta}(t) < \infty$. Assume that $f\in L_2(\Omega^t)$, $v(t_0)\in H^1(\Omega)$. Then for $\delta(t)$ sufficiently small $v \in W^{2,1}_2(\Omega^t)$ and
	\begin{equation*}
		\norm{v}_{W^{2,1}_2(\Omega^t)} + \norm{\nabla p}_{L_2(\Omega^t)} \leq c_{\alpha,\nu,\nu_r,I,P,\Omega}\left(E_{v,\omega}(t) + E_{h,\theta}(t) + \norm{f'}_{L_2(\Omega^t)} + \norm{v(t_0)}_{H^1(\Omega)} + 1\right)^3.
	\end{equation*}
\end{lem}

\begin{proof}
	Let us rewrite \eqref{p1}$_1$, \eqref{p2}$_{1,2}$ and \eqref{p7} in the following form
	\begin{equation*}
		\begin{aligned}
			&v_{,t} - (\nu + \nu_r)\triangle v + \nabla p = - v\cdot \nabla v + 2\nu_r\Rot\omega + f   & &\text{in } \Omega^t,\\
			&\Div v = 0 & &\text{in } \Omega^t,\\
			&v\cdot n = 0 & &\text{on } S^t,\\
			&\Rot v \times n = 0  & &\text{on } S^t, \\
			&v\vert_{t = t_0} = v(t_0) & &\text{on } \Omega\times\{t = t_0\}.
		\end{aligned}
	\end{equation*}
	For solution to the above problem we have the following estimate (see Lemma \ref{lem17})
	\begin{equation}\label{eq20}
		\norm{v}_{W^{2,1}_2(\Omega^t)} + \norm{\nabla p}_{L_2(\Omega^t)} 
		\leq c_{\Omega} \left(\norm{v\cdot \nabla v}_{L_2(\Omega^t)} + \norm{\Rot \omega}_{L_2(\Omega^t)} + \norm{f}_{L_2(\Omega^t)} + \norm{v(t_0)}_{H^1(\Omega)}\right).
	\end{equation}
	To estimate the first term on the right-hand side we use Lemma \ref{lem13}, which implies that for $\delta(t)$ small enough we have
	\begin{equation*}
		\norm{v}_{L_{10}(\Omega^t)} \leq c_{\alpha,\nu,\nu_r,I,P,\Omega} \left(E_{v,\omega}(t) + E_{h,\theta}(t) + \norm{f'}_{L_2(\Omega^t)} + \norm{v(t_0)}_{H^1(\Omega)}\right).
	\end{equation*}
	Subsequently, from the H\"older inequality we get
	\begin{equation*}
		\norm{v\cdot \nabla v}_{L_2(\Omega^t)} \leq c_I\norm{v}_{L_{10}(\Omega^t)} \norm{\nabla v}_{L_{\frac{5}{2}}(\Omega^t)}.
	\end{equation*}
	To estimate $\norm{\nabla v}_{L_\frac{5}{2}(\Omega^t)}$ we use the interpolation theorem (see Lemma \ref{lem10}) with $p = 2$, $q = \frac{5}{2}$. It gives
	\begin{equation*}
		\norm{\nabla v}_{L_{\frac{5}{2}}(\Omega^t)} \leq c_1 \epsilon^{\frac{1}{2}} \norm{v}_{W^{2,1}_2(\Omega^t)} + c_2 \epsilon^{-\frac{1}{2}} \norm{v}_{L_2(\Omega^t)}.
	\end{equation*}
	Setting $c_1\epsilon^{\frac{1}{2}} = \frac{1}{2c_I\norm{v}_{L_{10}(\Omega^t)}}$ yields
	\begin{equation*}
		\norm{v\cdot \nabla v}_{L_2(\Omega^t)} \leq \frac{1}{2}\norm{v}_{W^{2,1}_2(\Omega^t)} + 2c_1c_2c_I \norm{v}^2_{L_{10}(\Omega^t)} \norm{v}_{L_2(\Omega^t)}.
	\end{equation*}
	To estimate the second and the third term on the right-hand side in \eqref{eq20} we use Lemma \ref{l4}. Finally
	\begin{multline*}
		\norm{v}_{W^{2,1}_2(\Omega^t)} + \norm{\nabla p}_{L_2(\Omega^t)} \leq c_{\alpha,\nu,\nu_r,I,P,\Omega}\Big(E^3_{v,\omega}(t) + E^3_{h,\theta}(t) + \norm{f'}^3_{L_2(\Omega^t)} + \norm{v(t_0)}^3_{L_2(\Omega)} \\
		+ E_{v,\omega}(t) + \norm{v(t_0)}_{H^1(\Omega)}\Big).
	\end{multline*}
	This concludes the proof.
\end{proof}

\section{Existence of solutions}

Now we prove the existence of regular solutions. The proof is based on the fixed-point principle (see Lemma \ref{lem24}) and makes use of the inequalities we derived in previous Sections.

First we rewrite \eqref{p1} in following way

\begin{equation}\label{p10}
	\begin{aligned}
		&v_{,t}  - (\nu + \nu_r)\triangle v + \nabla p = - \lambda\left(\bar v\cdot \nabla \bar v + 2\nu_r\Rot\bar\omega\right) + f   & &\text{in } \Omega^t, \\
		&\Div v = 0 & &\text{in } \Omega^t,\\
		&\begin{aligned}
			&\omega_{,t} - \alpha \triangle \omega - \beta\nabla\Div\omega + 4\nu_r\omega \\
			&\mspace{60mu} = -\lambda \left(\bar v\cdot \nabla \bar \omega + 2\nu_r\Rot \bar v\right) + g
		\end{aligned}& &\text{in } \Omega^t,\\
		&\Rot v \times n = 0, \qquad v\cdot n = 0 & &\text{on } S^t, \\
		&\omega = 0 & &\text{on } S^t_1, \\
		&\omega' = 0, \qquad \omega_{3,x_3} = 0 & &\text{on } S^t_2, \\
		&v\vert_{t = t_0} = v(t_0), \qquad \omega\vert_{t = t_0} = \omega(t_0) & &\text{on } \Omega\times\{t = t_0\}.
	\end{aligned}
\end{equation}
where $\lambda\in[0,1]$ and $\bar v$, $\bar \omega$ are considered as given functions.

Let us introduce space
\begin{equation*}
	\mathfrak{M}(\Omega^t) := \left\{u\colon \norm{u}_{L_{\frac{20}{3}}(\Omega^t)} < \infty, \norm{\nabla u}_{L_{\frac{20}{7}}(\Omega^t)} < \infty\right\}.
\end{equation*}
Problem \eqref{p10} determines the mapping
\begin{align*}
	&\Phi\colon \mathfrak{M}(\Omega^t)\times\mathfrak{M}(\Omega^t)\times[0,1] \to \mathfrak{M}(\Omega^t)\times\mathfrak{M}(\Omega^t), \\
	&\Phi(\bar v,\bar \omega,\lambda) = (v,\omega).
\end{align*}
In Section \ref{sec3.2} we found a priori estimate for a~fixed point of $\Phi$ when $\lambda = 1$ (see Lemmas \ref{lem18} and \ref{lem25}). For $\lambda = 0$ we check if the uniqueness of solution is ensured. 
\begin{lem}\label{lem31}
	Suppose that $\lambda = 0$. Then, problem \eqref{p10} possesses a unique solution.
\end{lem}
\begin{proof}
	Let $(v^1,\omega^1)$ and $(v^2,\omega^2)$ be two different solutions to \eqref{p10}. Then, the pair $(V,\Theta)$, $V = v^1 - v^2$, $\Theta = \omega^1 - \omega^2$ is a solution to the problem
	\begin{equation}\label{eq36}
		\begin{aligned}
			&V_{,t} - (\nu + \nu_r) \triangle V + \nabla P = 0 & &\text{in $\Omega^t$}, \\
			&\Div V = 0 & &\text{in $\Omega^t$}, \\
			&\Theta_{,t} - \alpha \triangle \Theta - \beta\nabla \Div \Theta + 4\nu_r \Theta = 0 & &\text{in $\Omega$}, \\
			&\Rot V \times n = 0, \qquad V\cdot n = 0 & &\text{on } S^t, \\
			&\Theta = 0 & &\text{on } S^t_1, \\
			&\Theta' = 0, \qquad \Theta_{3,x_3} = 0 & &\text{on } S^t_2, \\
			&V\vert_{t = t_0} = 0, \qquad \Theta\vert_{t = t_0} = 0 & &\text{on } \Omega\times\{t = t_0\},
		\end{aligned}
	\end{equation}
	where we set $P = p^1 - p^2$. Multiplying the first equation by $V$, the third by $\Theta$ and integrating over $\Omega$ yields
	\begin{multline*}
		\Dt \int_{\Omega} \abs{V}^2 + \abs{\Theta}^2\, \ud x - (\nu + \nu_r)\int_{\Omega} \triangle V \cdot V\, \ud x - \alpha \int_{\Omega} \triangle \Theta \cdot \Theta\, \ud x - \beta \int_{\Omega} \nabla \Div \Theta \cdot \Theta\, \ud x \\
		+ \int_{\Omega} \nabla P \cdot V\, \ud x + 4\nu_r \norm{\Theta}^2_{L_2(\Omega)} = 0.
	\end{multline*}
	From Lemma \ref{l1} it follows
	\begin{align*}
		-&\int_{\Omega} \triangle V\cdot V\, \ud x = \int_{\Omega} \Rot\Rot V\cdot V\, \ud x = \int_{\Omega} \abs{\Rot V}^2\, \ud x + \int_S \Rot V \times n \cdot V\, \ud S = \norm{\Rot V}_{L_2(\Omega)}^2, \\
		-&\int_{\Omega} \triangle \Theta\cdot \Theta\, \ud x = \int_{\Omega} \Rot\Rot \Theta\cdot \Theta\, \ud x = \int_{\Omega} \abs{\Rot \Theta}^2\, \ud x - \int_S \Theta \times n \cdot \Rot \Theta\, \ud S = \norm{\Rot \Theta}_{L_2(\Omega)}^2,
	\end{align*}
	where the boundary integrals vanished due to \eqref{eq36}$_{4,5,6}$. Integration by parts yields
	\begin{equation*}
		-\int_{\Omega}\nabla \Div \Theta\cdot\Theta\, \ud x = \int_{\Omega} \abs{\Div \Theta}^2\, \ud x - \int_S \Div\Theta (\Theta\cdot n)\, \ud S = \norm{\Div \Theta}^2_{L_2(\Omega)},
	\end{equation*}
	where the boundary integral also vanishes due to \eqref{eq36}$_{5,6}$. Thus
	\begin{multline*}
		\Dt \int_{\Omega} \abs{V}^2 + \abs{\Theta}^2\, \ud x + (\nu + \nu_r)\norm{\Rot V}^2 + \alpha \left(\norm{\Rot \Theta}^2_{L_2(\Omega)} + \norm{\Div \Theta}^2_{L_2(\Omega)}\right) + \beta \norm{\Div \Theta}^2_{L_2(\Omega)} \\
		+ 4\nu_r \norm{\Theta}^2_{L_2(\Omega)} = 0.
	\end{multline*}
	After integrating with respect to $t \in (t_0,t_1)$ we obtain
	\begin{equation*}
		\norm{V(t)}_{L_2(\Omega)}^2 + \norm{\Theta(t)}^2_{L_2(\Omega)} \leq \norm{V(t_0)}^2_{L_2(\Omega)} + \norm{\Theta(t_0)}^2_{L_2(\Omega)} = 0.
	\end{equation*}
	Thus, we have proved the uniqueness. The existence of solution follow from Lemmas \ref{lem16} and \ref{lem17}.	This ends the proof.
\end{proof}

Next we show that $\Phi$ is a compact and continuous mapping.
\begin{lem}\label{lem23}
	The mapping $\Phi$ is compact and continuous.
\end{lem}
\begin{proof}
	Assume that $\bar v\in \mathfrak{M}(\Omega^t)$. Then
	\begin{equation*}
		\norm{\bar v\cdot \nabla \bar v}_{L_2(\Omega^t)} \leq \norm{\bar v}_{L_{\frac{20}{3}}(\Omega^t)}\norm{\nabla \bar v}_{L_{\frac{20}{7}}(\Omega^t)} \leq \norm{\bar v}_{\mathfrak{M}(\Omega^t)}^2.
	\end{equation*}
	In the same way we get that
	\begin{equation*}
		\norm{\bar v \cdot \nabla \bar \omega}_{L_2(\Omega^t)} \leq \norm{\bar v}_{L_{\frac{20}{3}}(\Omega^t)} \norm{\nabla \bar \omega}_{L_{\frac{20}{7}}(\Omega^t)} \leq \norm{\bar v}_{\mathfrak{M}(\Omega^t)}\norm{\bar \omega}_{\mathfrak{M}(\Omega^t)}.
	\end{equation*}
	In view of Lemmas \ref{lem16} and \ref{lem17} we get that the solution to \eqref{p10} belongs to $W^{2,1}_2(\Omega^t)$. From Lemma \ref{lem10} we deduce that the embeddings
	\begin{equation}\label{eq21}
		\begin{aligned}
			&W^{2,1}_2(\Omega^t) \hookrightarrow L_{\frac{20}{3}}(\Omega^t) \Leftrightarrow \kappa = 2 - 5\left(\frac{1}{2} - \frac{3}{20}\right) = \frac{1}{4} > 0, \\
			&W^{2,1}_2(\Omega^t) \hookrightarrow L_{\frac{20}{7}}(0,t;W^1_{\frac{20}{7}}) \Leftrightarrow \kappa = 2 - 1 - 5\left(\frac{1}{2} - \frac{7}{20}\right) = \frac{1}{4} > 0
		\end{aligned}
	\end{equation}
	are compact. This proves the compactness of $\Phi$.

	In order to prove the continuity we consider problem \eqref{p10} in a form
	\begin{equation}
		\begin{aligned}\label{p12}
			&\begin{aligned}
				&v^k_{,t} - (\nu + \nu_r)\triangle v^k + \nabla p^k \\
				&\mspace{60mu} = -\lambda\left(\bar v^k\cdot \nabla \bar v^k - 2\nu_r\Rot\bar\omega^k\right) + f
			&\end{aligned}& &\text{in } \Omega^t,\\
			&\begin{aligned}
				&\omega_{,t}^k - \alpha \triangle \omega - \beta\nabla\Div\omega^k + 4\nu_r \omega^k\\
				&\mspace{60mu} = -\lambda\left(\bar v^k\cdot\nabla \bar \omega^k  + 2\nu_r\Rot \bar v^k\right) + g
			\end{aligned}& &\text{in } \Omega^t,\\
			&\Div v^k = 0 & &\text{in } \Omega^t, \\
			&\Rot v^k\times n = 0, \qquad v^k \cdot n = 0 & &\text{on } S^t, \\
			&\omega^k = 0 & &\text{on } S_1^t, \\
			&\left(\omega^k\right)' = 0, \qquad \omega_{3,x_3}^k = 0 & &\text{on } S^t_2, \\
			&v^k\vert_{t = t_0} = v(t_0), \qquad \omega^k\vert_{t = t_0} = \omega(t_0) & &\text{on } \Omega\times\{t = t_0\},
		\end{aligned}
	\end{equation}
	where $k = 1,2$. Let
	\begin{equation*}
		V = v^1 - v^2, \qquad P = p^1 - p^2, \qquad \Theta = \omega^1 - \omega^2.
	\end{equation*}
	Then $(V,\Theta)$ is solution to the problem
	\begin{equation*}
		\begin{aligned}
			&V_{,t}  - (\nu + \nu_r)\triangle V + \nabla P = - \lambda\left(\bar V\cdot \nabla \bar v^1 + \bar v^2\cdot \nabla \bar V + 2\nu_r\Rot\bar\Theta\right)  & &\text{in } \Omega^t, \\
			&\Div V = 0 & &\text{in } \Omega^t,\\
			&\begin{aligned}
				&\Theta_{,t} - \alpha \triangle \Theta - \beta\nabla\Div\Theta  + 4\nu_r\Theta\\
				&\mspace{60mu} = -\lambda \left(\bar v^1\cdot \nabla \bar \Theta + \bar V\cdot \nabla \bar \omega^2  + 2\nu_r\Rot \bar V\right)
			\end{aligned}& &\text{in } \Omega^t,\\
			&\Rot V \times n = 0, \qquad V\cdot n = 0 & &\text{on } S^t, \\
			&\Theta = 0 & &\text{on } S^t_1, \\
			&\Theta' = 0, \qquad \Theta_{3,x_3} = 0 & &\text{on } S^t_2, \\
			&V\vert_{t = t_0} = 0, \qquad \Theta\vert_{t = t_0} = 0 & &\text{on } \Omega\times\{t = t_0\}.
		\end{aligned}
	\end{equation*}
	Assume that $\lambda \neq 0$. Then the estimates
	\begin{equation*}
		\norm{V}_{\mathfrak{M}(\Omega^t)} \leq c_{\Omega} \norm{V}_{W^{2,1}_2(\Omega^t)} \leq c_{\Omega}\lambda\left(\norm{\bar V\cdot \nabla \bar v^1 + \bar v^2\cdot \nabla \bar V + 2\nu_r\Rot\bar\Theta}_{L_2(\Omega^t)}\right)
	\end{equation*}
	and
	\begin{equation*}
		\norm{\Theta}_{\mathfrak{M}(\Omega^t)} \leq c_{\Omega}\norm{\Theta}_{W^{2,1}_2(\Omega^t)} \leq c_{\Omega} \lambda\left(\norm{\bar v^1\cdot \nabla \bar \Theta + \bar V\cdot \nabla\bar \omega^2 + 2\nu_r\Rot \bar V}_{L_2(\Omega^t)}\right).
	\end{equation*}
	follow from Lemmas \ref{lem17} and \ref{lem16}. Next we add both inequalities and utilize the H\"older inequality. This yields
	\begin{multline*}
		\norm{V}_{\mathfrak{M}(\Omega^t)} + \norm{\Theta}_{\mathfrak{M}(\Omega^t)} \leq c_{\nu_r,\lambda,\Omega}\Big(\norm{\bar V}_{L_{\frac{20}{3}}(\Omega^t)}\norm{\bar \nabla v^1}_{L_{\frac{20}{7}}(\Omega^t)} + \norm{\bar v^2}_{L_{\frac{20}{3}}(\Omega^t)}\norm{\nabla \bar V}_{L_{\frac{20}{7}}(\Omega^t)}\\
		+ \norm{\bar v^1}_{L_{\frac{20}{3}}(\Omega^t)}\norm{\nabla \bar\Theta}_{L_{\frac{20}{7}}(\Omega^t)} + \norm{\bar V}_{L_{\frac{20}{3}}(\Omega^t)}\norm{\nabla \bar \omega^2}_{L_{\frac{20}{7}}(\Omega^t)} + \norm{\Rot \bar \Theta}_{L_2(\Omega^t)} + \norm{\Rot \bar V}_{L_2(\Omega^t)}\Big).
	\end{multline*}
	Thus
	\begin{multline*}
		\norm{V}_{\mathfrak{M}(\Omega^t)} + \norm{\Theta}_{\mathfrak{M}(\Omega^t)} \leq c_{\nu_r,\lambda,\Omega}\Big(\norm{\bar V}_{\mathfrak{M}(\Omega^t)} + \norm{\bar \Theta}_{\mathfrak{M}(\Omega^t)}\Big) \cdot \Big(\norm{\nabla \bar v^1}_{L_{\frac{20}{7}}(\Omega^t)} + \norm{\bar v^2}_{L_{\frac{20}{3}}(\Omega^t)} \\
		+ \norm{\bar v^1}_{L_{\frac{20}{3}}(\Omega^t)} + \norm{\nabla \bar \omega^2}_{L_{\frac{20}{7}}(\Omega^t)} 		+ \norm{\Rot \bar \Theta}_{L_2(\Omega^t)} + \norm{\Rot \bar V}_{L_2(\Omega^t)}\Big).
	\end{multline*}
	In view of \eqref{eq21} and Lemmas \ref{lem25}, \ref{lem18} and \ref{l4} we can estimate all norms in the last bracket in terms of data only. This observation results in the following inequality
	\begin{equation*}
		\norm{V}_{\mathfrak{M}(\Omega^t)} + \norm{\Theta}_{\mathfrak{M}(\Omega^t)} \leq c_{\nu_r,\lambda,\Omega,\text{data}} \Big(\norm{\bar V}_{\mathfrak{M}(\Omega^t)} + \norm{\bar \Theta}_{\mathfrak{M}(\Omega^t)}\Big),
	\end{equation*}
	where $c_{\text{data}}$ indicates the dependence on the data.
	It proves the uniform continuity of $\Phi$. The continuity of $\Phi$ with respect to $\lambda$ is evident. This end the proof. 
\end{proof}

\begin{lem}\label{lem30}
	Let $E_{v,\omega}(t) < \infty$, $E_{h,\theta}(t) < \infty$. Assume that $f,g \in L_2(\Omega^t)$, $v(t_0), \omega(t_0) \in H^1(\Omega)$. Then, for $\delta(t)$ small enough problem \eqref{p1} admits a solution $(v,\omega) \in W^{2,1}_2(\Omega^t)\times W^{2,1}_2(\Omega^t)$ such that
	\begin{equation*}
		\norm{v}_{W^{2,1}_2(\Omega^t)} + \norm{\nabla p}_{L_2(\Omega^t)} \leq c_{\alpha,\nu,\nu_r,I,P,\Omega}\left(E_{v,\omega}(t) + E_{h,\theta}(t) + \norm{f'}_{L_2(\Omega^t)} + \norm{v(t_0)}_{H^1(\Omega)} + 1\right)^3
	\end{equation*}
	and
	\begin{equation*}
		\norm{\omega}_{W^{2,1}_{2}(\Omega^t)} \leq c_{\alpha,\nu,\nu_r,I,P,\Omega}\left(E_{v,\omega}(t) + E_{h,\theta}(t) + \norm{f'}_{L_2(\Omega^t)} + \norm{v(t_0)}_{H^1(\Omega)} + \norm{\omega(t_0)}_{H^1(\Omega)} + 1\right)^3.
	\end{equation*}
\end{lem}
\begin{proof}
	Lemmas \ref{lem30} and \ref{lem23} ensure that the assumptions of the Leary-Schauder theorem (see Lemma \ref{lem24}) are met. Hence, we obtain the existence of solution $(v,\omega) \in W^{2,1}_2(\Omega^t)\times W^{2,1}_2(\Omega^t)$ and by Lemma \ref{lem17} also the existence of $\nabla p \in L_2(\Omega^t)$. The estimates follow from Lemmas \ref{lem25} and \ref{lem18}.

\end{proof}

\section{Uniqueness of regular solutions}

In the previous Section we have proved the existence of regular solutions to problem \eqref{p1}. Now we would like to ask about their uniqueness. Like for the ordinary Navier-Stokes equations, the answer is positive (see Lemma below).

\begin{lem}\label{lem26}
	Suppose that $\delta(t)$ is small enough. Then, problem \eqref{p1} admits a unique solution.
\end{lem}
\begin{proof}
	The proof is straightforward and is based on the Gronwall inequality for difference of solutions.

	Let $(v^1,\omega^1)$ and $(v^2,\omega^2)$ be two different solutions to problem \eqref{p1}. Let us denote $V = v^1 - v^2$, $\Theta = \omega^1 - \omega^2$ and $P = p^1 - p^2$. Then, the pair $(V,\Theta)$ is a solution to the problem
	\begin{equation}\label{eq25}
		\begin{aligned}
				&V_{,t}  - (\nu + \nu_r)\triangle V + \nabla P = - V\cdot \nabla v^1 - v^2\cdot \nabla V + 2\nu_r\Rot\Theta & &\text{in } \Omega^t, \\
			&\Div V = 0 & &\text{in } \Omega^t,\\
			&\begin{aligned}
				&\Theta_{,t} - \alpha \triangle \Theta - \beta\nabla\Div\Theta + 4\nu_r \Theta \\
				&\mspace{60mu} = -v^1\cdot \nabla \Theta - V\cdot \nabla \omega^2 + 2\nu_r\Rot V
			\end{aligned}& &\text{in } \Omega^t,\\
			&\Rot V \times n = 0, \qquad V\cdot n = 0 & &\text{on } S^t, \\
			&\Theta = 0 & &\text{on } S^t_1, \\
			&\Theta' = 0, \qquad \Theta_{3,x_3} = 0 & &\text{on } S^t_2, \\
			&V\vert_{t = t_0} = 0, \qquad \Theta\vert_{t = t_0} = 0 & &\text{on } \Omega\times\{t = t_0\}.
		\end{aligned}
	\end{equation}
	Now we multiply the first equation by $V$ and the third by $\Theta$. Next we integrate and since $V$ and $\Theta$ satisfy the boundary conditions \eqref{p2} and the assumption $(\mathbf{A})$ from Lemma \ref{l1} we easily see that
	\begin{multline*}
		\frac{1}{2}\Dt \int_{\Omega} \abs{V}^2 + \abs{\Theta}^2\, \ud x + (\nu + \nu_r) \norm{\Rot V}^2_{L_2(\Omega)} + \alpha \left(\norm{\Rot \Theta}^2_{L_2(\Omega)} + \norm{\Div \Theta}^2_{L_2(\Omega)}\right) \\
	+ \beta \norm{\Div \Theta}^2_{L_2(\Omega)} + 4\nu_r\norm{\Theta}^2_{L_2(\Omega)}	\\
		= -\int_{\Omega} V\cdot \nabla v^1 \cdot V\, \ud x + 4\nu_r \int_{\Omega} \Rot V \cdot \Theta\, \ud x + \int_{\Omega} V\cdot \nabla \omega^2 \cdot \Theta\, \ud x.
	\end{multline*}
	From the H\"older and the Young inequalities it follows that
	\begin{align*}
		-\int_{\Omega} V\cdot \nabla v^1 \cdot V\, \ud x &\leq \epsilon_1 \norm{V}^2_{L_6(\Omega)} + \frac{1}{4\epsilon_1} \norm{\nabla v^1}^2_{L_3(\Omega)}\norm{V}^2_{L_2(\Omega)}, \\
		4\nu_r \int_{\Omega} \Rot V\cdot \Theta\, \ud x &\leq 4\nu_r \epsilon_2 \norm{\Rot V}^2_{L_2(\Omega)} + \frac{\nu_r}{\epsilon_2} \norm{\Theta}^2_{L_2(\Omega)}, \\
		\int_{\Omega} V\cdot \nabla \omega^2 \cdot \Theta\, \ud x &\leq \epsilon_3 \norm{V}^2_{L_6(\Omega)} + \frac{1}{4\epsilon_3} \norm{\nabla \omega^2}^2_{L_3(\Omega)}\norm{\Theta}^2_{L_2(\Omega)}.
	\end{align*}
	Now we set $\epsilon_2 = \frac{1}{4} \Rightarrow \frac{\nu_r}{\epsilon_2} = 4\nu_r$ and utilize Lemma \ref{l2}, which implies that
	\begin{equation*}
		\frac{\nu}{c_{\Omega}} \norm{V}^2_{H^1(\Omega)} \leq \nu\norm{\Rot V}^2_{L_2(\Omega)}.
	\end{equation*}
	From the embedding $H^1 \hookrightarrow L_6$ and for $\epsilon_1 = \epsilon_3 = \frac{\nu}{4c_{\Omega}c_I}$ we deduce that
	 \begin{equation*}
		\frac{1}{2}\Dt \int_{\Omega} \abs{V}^2 + \abs{\Theta}^2\, \ud x + \frac{\nu}{2c_{\Omega}}\norm{\Rot V}^2_{L_2(\Omega)} \leq \frac{c_{I,\Omega}}{\nu} \left(\norm{\nabla v^1}^2_{L_3(\Omega)} \norm{V}^2_{L_2(\Omega)} + \norm{\nabla \omega^2}^2_{L_3(\Omega)} \norm{\Theta}^2_{L_2(\Omega)}\right)
	\end{equation*}
	From the Gronwall inequality we infer that
	\begin{multline*}
		\sup_{t_0 \leq t \leq t_1}\left(\norm{V(t)}^2_{L_2(\Omega)} + \norm{\Theta(t)}^2_{L_2(\Omega)}\right) \\
		\leq c_{\nu,I,\Omega}\exp\left(\norm{\nabla v^1}^2_{L_2(t_0,t_1;L_3(\Omega))} + \norm{\nabla \omega^2}^2_{L_2(t_0,t_1;L_3(\Omega))}\right) \left(\norm{V(t_0)}^2_{L_2(\Omega)} + \norm{\Theta(t_0)}^2_{L_2(\Omega)}\right).
	\end{multline*}
	Since $H^2(\Omega) \hookrightarrow W^1_6(\Omega) \hookrightarrow W^1_3(\Omega)$ we see that
	\begin{align*}
		\norm{\nabla v^1}^2_{L_2(t_0,t_1;L_3(\Omega))} &\leq c_{\Omega} \norm{v^1}^2_{W^{2,1}_2(\Omega^t)}, \\
		\norm{\nabla \omega^2}^2_{L_2(t_0,t_1;L_3(\Omega))} &\leq c_{\Omega}\norm{\omega^2}^2_{W^{2,1}_2(\Omega^t)},
	\end{align*}
	which justifies that the right-hand side is finite. In our case $V(t_0) = \Theta(t_0) = 0$. This implies that $V(t) = \Theta(t) \equiv 0$, which proves the uniqueness of solutions.
\end{proof}

\section{Proof of Theorem 1}

\begin{proof}[Proof of Theorem \ref{t1}]
	The existence of solutions follows from Lemma \ref{lem30}, whereas uniqueness from Lemma \ref{lem26}.
\end{proof}

\section{Final remarks}

We emphasize that the proof of the existence of regular solutions is free of constants that depend on time. As it can be regarded as a proof of global in time regular solutions because there are no restriction on $t$. If we have a solution on $(0,t)$ we can always extend it to $(0,t + 1)$. But we cannot simply put $t_0 = 0$, $t_1 = \infty$ because we would be confronted with improper integrals with respect to time. Among other things we adopt another approach which will allow us to consider the interval $(0,\infty)$ but we demonstrate it in forthcoming paper.

The author wishes to express his thanks to Professor Wojciech Zaj\k{a}czkowski for many stimulating conversations during preparation of this work.


\begin{thebibliography}{99}

\bibitem[Ala05]{ala}
W.~Alame, \emph{On existence of solutions for the nonstationary stokes system
  with boundary slip conditions.}, Appl. Math. \textbf{32} (2005), no.~2,
  195--223.

\bibitem[BDRM10]{bol}
J.L. Boldrini, M.~Dur\'an, and M.A. Rojas-Medar, \emph{Existence and uniqueness
  of strong solution for the incompressible micropolar fluid equations in
  domains of $\mathbb{R}^{3}$.}, Ann. Univ. Ferrara, Sez. VII, Sci. Mat.
  \textbf{56} (2010), no.~1, 37--51.

\bibitem[BZ97]{bur}
M.~Burnat and W.M. Zaj{\k a}czkowski, \emph{{On local motion of a compressible
  barotropic viscous fluid with the boundary slip condition.}}, Topol. Methods
  Nonlinear Anal. \textbf{10} (1997), no.~2, 195--223.

\bibitem[CMR98]{Clopeau:1998vj}
T.~Clopeau, A.~Mikeli{\'c}, and R.~Robert, \emph{On the vanishing viscosity
  limit for the 2d incompressible navier-stokes equations with the friction
  type boundary conditions}, Nonlinearity \textbf{11} (1998), no.~6,
  1625--1636.

\bibitem[DH82]{daf}
C.M. Dafermos and L.~Hsiao, \emph{Global smooth thermomechanical processes in
  one-dimensional nonlinear thermoviscoelasticity.}, Nonlinear Anal., Theory
  Methods Appl. \textbf{6} (1982), 435--454.

\bibitem[DL72]{duv}
G.~Duvaut and J.L. Lions, \emph{{Les in\'equations en m\'ecanique et en
  physique.}}, {Dunod; 1st edition}, 1972 (French), English translation in:
  Inequalities in Mechanics \& Physics, Springer; 1st edition, 1976.

\bibitem[Eri66]{erin}
A.C. Eringen, \emph{{Theory of micropolar fluids.}}, J. Math. Mech. \textbf{16}
  (1966), 1--16.

\bibitem[GR77]{gald}
G.P. Galdi and S.~Rionero, \emph{{A note on the existence and uniqueness of
  solutions of the micropolar fluid equations.}}, Int. J. Eng. Sci. \textbf{15}
  (1977), 105--108.

\bibitem[Kel06]{kell}
J.P. Kelliher, \emph{Navier-stokes equations with navier boundary conditions
  for a bounded domain in the plane.}, SIAM J. Math. Anal. \textbf{38} (2006),
  no.~1, 210--232.

\bibitem[LM68]{lions}
J.L. Lions and E.~Magenes, \emph{Probl\`emes aux limites non homogenes et
  applications. vol. 1.}, Dunod, 1968 (French), English translation in:
  Non-homogeneous boundary value problems and applications. Vol. II. Translated
  from the French by P. Kenneth., 1972.

\bibitem[LSU67]{lad}
O.A. Lady{\v z}enskaja, V.A. Solonnikov, and N.N. Ural'ceva, \emph{Linear and
  quasilinear equations of parabolic type}, Translated from the Russian by S.
  Smith. Translations of Mathematical Monographs, Vol. 23, American
  Mathematical Society, Providence, R.I., 1967.

\bibitem[{\L}uk89]{luk11}
G.~{\L}ukaszewicz, \emph{On the existence, uniquenss and asymptotic properties
  for solutions of flows of asymmetric fluids.}, Rend. Accad. Naz. Sci. Detta
  XL, V. Ser. \textbf{13} (1989), no.~1, 105--120.

\bibitem[{\L}uk99]{luk1}
\bysame, \emph{Micropolar fluids. theory and applications.}, Boston:
  Birkh{\"a}user, 1999.

\bibitem[{\L}uk01]{luk2}
\bysame, \emph{{Long-time behavior of 2D micropolar fluid flows.}}, Math.
  Comput. Modelling \textbf{34} (2001), no.~5-6, 487--509.

\bibitem[Nav27]{nav}
C.M.L.H. Navier, \emph{{Sur les lois de l'equilibre et du mouvement des corps
  \'elastiques.}}, Acad. R. Sci. Inst. France \textbf{6} (1827).

\bibitem[OTRM97]{ort}
E.E. Ortega-Torres and M.A. Rojas-Medar, \emph{Magneto-micropolar fluid motion:
  Global existence of strong solutions.}, Abstr. Appl. Anal. \textbf{4} (1997),
  no.~2, 109--125.

\bibitem[Pop69]{pop1}
A.S. Popel, \emph{{On the hydrodynamics of suspensions.}}, Izv. AN SSSR
  \textbf{4} (1969), 24--30 (Russian).

\bibitem[PRU74]{pop2}
A.S. Popel, S.A Regirer, and P.I. Usick, \emph{{A continuum model of blood
  flow.}}, Biorheology \textbf{11} (1974), 427--437.

\bibitem[RMB98]{roj1}
M.A. Rojas-Medar and J.L. Boldrini, \emph{Magneto-micropolar fluid motion:
  Existence of weak solutions.}, Rev. Mat. Complut. \textbf{11} (1998), no.~2,
  443--460.

\bibitem[RMOT05]{RojasMedar:2005cv}
M.A. Rojas-Medar and E.E. Ortega-Torres, \emph{The equations of a viscous
  asymmetric fluid: An interactive approach}, Z. Angew. Math. Mech. \textbf{85}
  (2005), no.~7, 471--489.

\bibitem[RZ08]{ren1}
J.~Renc{\l}awowicz and W.M. Zaj{\k a}czkowski, \emph{Large time regular
  solutions to the navier-stokes equations in cylindrical domains.}, Topol.
  Methods Nonlinear Anal. \textbf{32} (2008), no.~1, 69--87.

\bibitem[Sol65]{sol2}
V.A. Solonnikov, \emph{On boundary value problems for linear parabolic systems
  of differential equations of general form.}, Proc. Steklov Inst. Math.
  \textbf{83} (1965), 1--184, Translation from: Trudy Mat. Inst. Steklov 83
  (1965), 3--163, (Russian).

\bibitem[Sol73]{sol1}
\bysame, \emph{Overdetermined elliptic boundary-value problems.}, Journal of
  Mathematical Sciences \textbf{1} (1973), no.~4, 477--512, Translation from:
  Zap. Nauchn. Sem. LOMI 21 (1971), 112--158, (Russian).

\bibitem[Tem79]{tem}
R.~Temam, \emph{Navier-{S}tokes equations.}, revised ed., Studies in
  Mathematics and its Applications, vol.~2, North-Holland Publishing Co.,
  Amsterdam, 1979, Theory and numerical analysis, With an appendix by F.
  Thomasset.

\bibitem[vW85]{wah}
W.~von Wahl, \emph{{The equations of Navier-Stokes equations and abstract
  parabolic equations.}}, {Friedrick Vieweg \& Son}, 1985.

\bibitem[WZ05]{wm4}
M.~Wiegner and W.M. Zaj{\k a}czkowski, \emph{On stability of axially symmetric
  solutions to {N}avier-{S}tokes equations in a cylindrical domain and with
  boundary slip conditions}, Regularity and other aspects of the
  {N}avier-{S}tokes equations, Banach Center Publ., vol.~70, Polish Acad. Sci.,
  Warsaw, 2005, pp.~251--278.

\bibitem[Yam05]{yam}
N.~Yamaguchi, \emph{Existence of global strong solution to the micropolar fluid
  system in a bounded domain.}, Math. Methods Appl. Sci. \textbf{28} (2005),
  no.~13, 1507--1526.

\bibitem[Zaj04]{wm5}
W.M. Zaj{\k a}czkowski, \emph{Global special regular solutions to the
  navier-stokes equations in a cylindrical domain without the axix of
  symmetry.}, Topol. Methods Nonlinear Anal. \textbf{24} (2004), no.~1,
  69--105.

\bibitem[Zaj05]{wm2}
\bysame, \emph{Long time existence of regular solutions to navier-stokes
  equations in cylindrical domains under boundary slip conditions.}, Stud.
  Math. \textbf{169} (2005), no.~3, 243--285.

\bibitem[Zaj09]{wm7}
\bysame, \emph{{Global regular solutions to the Navier-Stokes equations in an
  axially symmetric domain.}}, Topol. Methods Nonlinear Anal. \textbf{33}
  (2009), no.~2, 233--274.

\bibitem[Zaj11]{wm6}
\bysame, \emph{{On global regular solutions to the Navier-Stokes equations in
  cylindrical domains.}}, Topol. Methods Nonlinear Anal. \textbf{37} (2011),
  no.~1, 55--86.

\end{thebibliography}
\end{document}